\documentclass[12pt,amsymb,fullpage]{amsart}
\usepackage{amssymb,amscd,pstricks}

\newtheorem{theorem}{Theorem}[section]
\newtheorem{defn}[theorem]{Definition}

\newtheorem{lemma}[theorem]{Lemma}

\newtheorem{eple}[theorem]{Example}
\newtheorem{rmk}[theorem]{Remarks}
\newtheorem{dsc}[theorem]{Discussion}
\newtheorem{nota}[theorem]{Notation}

\newsavebox{\indbin}
\savebox{\indbin}{\begin{picture}(0,0)
\newlength{\gnu}
\settowidth{\gnu}{$\smile$} \setlength{\unitlength}{.5\gnu}
\put(-1,-.65){$\smile$} \put(-.25,.1){$|$}
\end{picture}}

\newcommand{\be}{\begin{enumerate}}
\newcommand{\bd}{\begin{defn}}
\newcommand{\bt}{\begin{theorem}}
\newcommand{\bl}{\begin{lemma}}
\newcommand{\ee}{\end{enumerate}}
\newcommand{\ed}{\end{defn}}
\newcommand{\et}{\end{theorem}}
\newcommand{\el}{\end{lemma}}

\begin{document}
\title{A Theory of Harmonic Variations}
\author{Tristram de Piro}
\address{Flat 1, 98 Prestbury Road, Cheltenham, GL52 3BG. Research Commons Library, University of Exeter, Exeter, EX4 4SB}
\email{t.de-piro@exeter.ac.uk}
\thanks{Thanks to Julius Plucker}
\begin{abstract}
We consider a class of "harmonic variations" for nonsingular curves, obtained as asymptotic degenerations along bitangents. On a geometric level, we obtain an attractive relationship between the class and the genus of $C$. The distribution of class points in pairs across nonsingular curves with such variations, further suggests applications to understanding covalent bonding in terms of shared electrons.
\end{abstract}
\maketitle
\begin{section}{Alcoves and Class Formulas}
\label{alcoves}

Let $n$ be an odd number, and $C$ a circle of radius $1$, centred about the origin $(0,0)$, of a real coordinate system $(x,y)$. Suppose that a regular $n$-sided polygon is inscribed inside the circle, with vertices $\{p_{0},\ldots,p_{j},\ldots,p_{n-1}\}$, with coordinates $e^{{2\pi ij\over n}}$ and $\{l_{0},\ldots,l_{j},\ldots,l_{n-1}\}$ are the lines formed by the edges of the polygon, so that $l_{j}$ passes through the pair of vertices $\{p_{j},p_{j+1}\}$, mod$(n)$. By construction, the $n$ intersections
$(l_{j}\cap l_{j+1})$, for $0\leq j\leq n-1$, mod$(n)$, lie on the unit circle. We claim, more generally, that;\\

\begin{lemma}
\label{equal}
If $1\leq k\leq {n-1\over 2}$, the $n$ intersections $(l_{j}\cap l_{j+k})$, mod$(n)$, lie on a circle, centred about $(0,0)$,  of radius ${sin({\pi\over 2}(1-{2\over n}))\over sin({\pi\over 2}(1-{2k\over n}))}$, with equal angles subtended by consecutive pairs to the origin $(0,0)$.
\end{lemma}

\begin{proof}
For convenience of notation, let $O$ denote the origin $(0,0)$, and let $C$ denote the intersection $(l_{j}\cap l_{j+k})$. Let $\{\alpha,\beta,\gamma\}$ denote the angles $\{Op_{j}C,p_{j}OC,p_{j}CO\}$ of the triangle with vertices $\{p_{j},O,C\}$, let $\delta$ denote the angle $p_{j}Op_{j+1}$ of the triangle with vertices $\{p_{j},O,p_{j+1}\}$, let $\epsilon$ denote the angle between the lines $l_{j}$ and $l_{j+1}$ and $r$ the length of the edge $OC$. We have that;\\

$\delta=\epsilon={2\pi\over n}$\\

$\alpha={\pi-\epsilon\over 2}={\pi\over 2}(1-{2\over n})$\\

$\beta={\delta(k+1)\over 2}={\pi(k+1)\over n}$\\

$\gamma=\pi-(\alpha+\beta)={\pi\over 2}(1-{2k\over n})$\\

By the sine rule, applied to the triangle $p_{j}OC$, we have that;\\

$r={sin(\alpha)\over sin(\gamma)}={sin({\pi\over 2}(1-{2\over n}))\over sin({\pi\over 2}(1-{2k\over n}))}$\\

as required. The last claim follows easily from calculating the angle $C_{1}OC_{2}={2\pi\over n}$, for two consecutive intersections in the set $l_{j}\cap l_{j+k}$.

\end{proof}

\begin{rmk}
\label{attractive}
It follows that all of the $C^{n}_{2}$ intersections between the $n$ lines $\{l_{0},\ldots,l_{j},\ldots,l_{n-1}\}$ lie on concentric circles about the origin $O$. The pattern of lines and intersections forms an attractive radiating pattern, harmoniously arranged in the plane.

\end{rmk}

If $n$ is an even number, we can perform the same construction, but obtain a slightly modified version of the previous lemma;\\

\begin{lemma}
\label{intersections}
Let hypotheses and notation be as above. If $1\leq k\leq {n-2\over 2}$, the $n$ intersections $(l_{j}\cap l_{j+k})$ are arranged as in Lemma \ref{equal}. If $k={n\over 2}$, there exist ${n\over 2}$ intersections in the set $l_{j}\cap l_{j+k}$, situated on the circle at $\infty$, in the real projective plane ${\mathcal R}P^{2}$.

\end{lemma}

\begin{proof}

It is sufficient to observe that, when $k={n\over 2}$, the lines $l_{j}$ and $l_{j+k}$, mod$(n)$, are parallel, in the plane with affine coordinates $(x,y)$. Embedding the real affine plane in the projective plane ${\mathcal R}P^{2}$, we obtain ${n\over 2}$
intersections between the ${n\over 2}$ pairs of parallel lines in the set $\{l_{0},\ldots,l_{j},\ldots,l_{n-1}\}$, as the pair gradients are distinct.

\end{proof}

\begin{defn}
\label{graph}
We say that $n$ lines in the real projective plane ${\mathcal R}P^{2}$ are in general position, if no three of the lines intersect in a point. We say that the lines are in bounded position, if all of the intersections lie in the affine plane ${\mathcal R}^{2}$. To $n$ lines $\{l_{0},\ldots,l_{n-1}\}$ in bounded general position, we can associate a graph $G$, whose \emph{vertices} consist of the $C^{n}_{2}$ intersections of the lines in the affine plane. We say that two vertices $\{v_{1},v_{2}\}\subset G$ are connected by an \emph{edge} if;\\

 (i). There exists a line in the set $\{l_{0},\ldots,l_{n-1}\}$ containing $v_{1}$ and $v_{2}$.\\

 (ii). There does not exist a third vertex $v_{3}$, lying between $v_{1}$ and $v_{2}$, on the same line.\\

We define an edge to be the closed line segment, connecting two such vertices. We define an \emph{alcove} of the graph $G$ by the following properties;\\

 (i). A compact convex connected subset $V$ of ${\mathcal R}^{2}$.\\

 (ii). The boundary $\delta V$ is a union of edges, belonging to distinct lines.\\

 (iii). $V$ does not contain a proper subset $W$, satisfying properties $(i)$ and $(ii)$.\\

\end{defn}

We now show that;\\

\begin{lemma}
\label{convex}
If $\{v_{0},\ldots,v_{j},\ldots,v_{n-1}\}$ are vertices, connected by edges $\{e_{0},\ldots,e_{j},\ldots,e_{n-1}\}$, lying on distinct lines, forming a convex $n$-polygon $V$, then $V$ is an alcove.
\end{lemma}

\begin{proof}
  Suppose not, then clearly condition $(iii)$ fails. We can, therefore,  find a proper subset $W\subset V$, satisfying conditions $(i)$ and $(ii)$. Let $e$ be one of the edges of $W$, belonging to a line $l$. Suppose the edges of $V$ belong to lines $\{l_{0},\ldots,l_{j},\ldots,l_{n-1}\}$ respectively. If, $l$ does not coincide with one of these lines, then, by the definition of general position, it must intersect one of them in a vertex, distinct from $\{v_{1},\ldots,v_{j},\ldots,v_{n-1}\}$, on $\delta V$. This contradicts the fact that $\{e_{0},\ldots,e_{j},\ldots,e_{n-1}\}$ are edges. It follows, that $l$ must coincide with one of the lines $\{l_{0},\ldots,l_{j},\ldots,l_{n-1}\}$, say $l_{0}$. If $e$ does not coincide with the edge $e_{0}$, then $l_{0}$ must contain a point in the interior $(V\setminus\delta V)$ of $V$. It is straightforward to show that this contradicts the assumption that $V$ is convex, (\footnote{Let $x$ be the interior point, and let $\{v_{0},v_{1},v_{2}\}$ be the vertices, connecting the edges $e_{0}$ and $e_{1}$.  By convexity, the triangles with vertices $\{x,v_{1},v_{2}\}$ and $\{v_{0},v_{1},v_{2}\}$ both lie inside $V$. This implies that the interior of the edge $e_{1}$ is interior to $V$, contradicting the fact that $e_{1}$ is contained in $\delta V$.}). It follows that the boundary $\delta W\subset \delta V$. As $\delta W$ is connected, $\delta W=\delta V$, hence,  $W$ must coincide with $V$, showing the result.
\end{proof}

We make the following definition;\\

\begin{defn}
\label{vertex}
If $\{v_{0},\ldots,v_{j},\ldots,v_{n-1}\}$ are vertices, lying on distinct lines, forming a convex $n$-polygon $V$, we define the vertex number of $V$, to be the number of vertices of the lines in bounded general position, either interior to $V$ or interior to the line segments forming the boundary $\delta V$ of $V$.
\end{defn}

\begin{rmk}
\label{polygon}
By the previous lemma, a convex $n$-polygon, with vertex number $0$, is an alcove.
 \end{rmk}

We then have that;\\

\begin{lemma}
\label{side}

If an edge $e$ forms one side of a convex $n$-polygon $V$, which is not an alcove, having vertex number $m>0$, then $e$ forms one side of a convex, at most $n+1$-sided polygon $W$, having vertex number $0\leq r<m$.
\end{lemma}

\begin{proof}
Let $\{l_{0},\ldots,l_{j},\ldots,l_{n-1}\}$ enumerate the line segments, forming the boundary of $V$, with $e$ corresponding to $l_{0}$. As $V$ is not an alcove, there exists a vertex $v$, interior to one of these lines, on the boundary $\delta V$ of $V$. As $e$ is an edge, it cannot be interior to $l_{0}$, say $v$ is interior to $l_{1}$.  As $V$ is convex, there exists a new line $l_{n}$, passing through $v$. As the lines are in general position, and $e$ is an edge, the line $l_{n}$ must intersect the interior of one of the lines $\{l_{2},\ldots,l_{j},\ldots,l_{n-1}\}$, say $l_{j}$. It is easily checked that the polygon $W$, formed by the sides $\{e,l_{1},l_{n},l_{j},\ldots,l_{n-1}\}$ is convex, with $(n-j+3)$ sides, having vertex number $r<m$.
\end{proof}

As a straightforward consequence, we have that;\\

\begin{lemma}
\label{boundary}
Every edge $e$ lies on the boundary of at least one alcove.
\end{lemma}
\begin{proof}

 Suppose that the edge $e$ has vertices $v_{0}$ and $v_{1}$, belonging to a line $l_{0}$, let $l_{1}$ and $l_{2}$ be further lines containing these vertices, respectively, intersecting in a vertex $v_{2}$. If the triangle with vertices $\{v_{0},v_{1},v_{2}\}$ is an alcove, the result is shown. Otherwise, it satisfies the hypotheses of the previous lemma; one may then apply the result inductively, together with the fact that the number of vertices are finite, and the previous remark, to obtain the same result.

\end{proof}

The following results use a different argument;\\

\begin{lemma}
\label{intersection}
If $\{V_{1},V_{2}\}$ are two distinct alcoves, then $(V_{1}\cap V_{2})\subset (\delta V_{1}\cup\delta V_{2})$ and, consists of an edge or a vertex.
\end{lemma}
\begin{proof}
We clearly have that $(V_{1}\cap V_{2})$ satisfies condition $(i)$ in the definition of an alcove, and the boundary $\delta(V_{1}\cap V_{2})$ is contained in $(\delta V_{1}\cup\delta V_{2})$. Suppose that $(V_{1}\cap V_{2})$ contains an open subset of $\mathcal{R}^{2}$, $(*)$ then the boundary $\delta(V_{1}\cap V_{2})$ is connected, therefore, must consist either of a vertex, or a union of line segments. That the line segments form edges, follows from the fact that the lines forming the boundary $\delta V_{1}$ of $V_{1}$, intersect the lines forming the boundary $\delta V_{2}$ of $V_{2}$, in vertices. It follows that $(V_{1}\cap V_{2})$ also satisfies condition $(ii)$ in the definition of an alcove. As the intersection is a proper subset of both $V_{1}$ and $V_{2}$, this contradicts the assumption that $V_{1}$ and $V_{2}$ are both alcoves. It follows that $(*)$ fails, that is $(V_{1}\cap V_{2})\subset (\delta V_{1}\cup\delta V_{2})$. As $(V_{1}\cap V_{2})$ is connected, being convex, the intersection consists of a vertex or an edge, as required.
\end{proof}

\begin{lemma}
\label{two}
Every edge $e$ lies on the boundary of at most two alcoves.
\end{lemma}

\begin{proof}
Suppose that $e$ lies on the boundary of three distinct alcoves $\{V_{1},V_{2},V_{3}\}$. It is easily checked, that, if $x$ is an interior point of the edge $e$, then $x$ is an interior point of the union of alcoves $(V_{1}\cup V_{2})$. It follows that the intersection of $V_{3}$ with either $V_{1}$ or $V_{2}$, must contain an open subset of $\mathcal{R}^{2}$. This contradicts the previous result.
\end{proof}

\begin{lemma}
\label{genus}
If $n\geq 3$, there exist ${(n-1)(n-2)\over 2}$ alcoves, associated to the graph of Definition \ref{graph}.

\end{lemma}

\begin{proof}
When $n=3$, it is easily checked that there is $1$ alcove, formed by the vertices $\{v_{1},v_{2},v_{3}\}$ of a triangle, obtained from the intersection of three lines $\{l_{1},l_{2},l_{3}\}$ in bounded general position. We assume, inductively, that the result is true for $n$ lines in bounded general position. Let $l_{n+1}$ be a new line, added to $n$ lines $\{l_{1},\ldots,l_{j},\ldots,l_{n}\}$ in bounded general position. This introduces $n$ new vertices $\{v_{1},\ldots,v_{n}\}$ and $(n-1)$ new edges, corresponding to line segments $e_{j}$ between the vertices $v_{j}$ and $v_{j+1}$. We claim that an edge $e_{j}$ is on the boundary of two alcoves in the graph $G_{n+1}$ of $(n+1)$ lines, $(*)$, if and only if it passes through the interior of an alcove in the graph $G_{n}$ of $n$ lines, $(**)$. For assume that $(*)$ holds, and $e_{j}$ lies on the boundary of two alcoves $V_{1}$ and $V_{2}$. By Lemma 0.10, $e_{j}=(V_{1}\cap V_{2})$. Let $\{e_{j},f_{1},\ldots,f_{r}\}$ and $\{e_{j},g_{1},\ldots,g_{s}\}$ enumerate the consecutive edges of the alcoves $V_{1}$ and $V_{2}$ respectively. By the definition of lines in general position, the edges $\{f_{r},g_{1}\}$ and $\{f_{1},g_{s}\}$ belong to the same lines $l_{1}$ and $l_{2}$ respectively, in particular, either the polygon defining the boundary $\delta V_{1}$ or the polygon defining the boundary
$\delta V_{2}$ is inscribed within the triangles, having either edges $\{e_{j},f_{1},f_{r}\}$ or $\{e_{j},g_{1},g_{s}\}$. In either case, it follows that the union of alcoves $(V_{1}\cup V_{2})$ is convex. We now relabel the boundary $\delta(V_{1}\cup V_{2})$, after removing the edge $e_{j}$, consecutively, as $\{h_{1},f_{2},\ldots,f_{r-1},h_{2},g_{2},\ldots,g_{s-1}\}$, where $h_{1}$ are $h_{2}$ are the new edges in the graph $G_{n}$, obtained by joining the edges $\{f_{r},g_{1}\}$ and $\{f_{1},g_{s}\}$ in the graph $G_{n+1}$. As all the faces are edges, belonging to distinct lines, from the above, it follows that $(V_{1}\cup V_{2})$ is an alcove, by Lemma 0.5, hence $(**)$ follows. Conversely, assume that $(**)$ holds, and $e_{j}$ passes through the interior of an alcove $V$ in $G_{n}$. By the definition of an alcove, the boundary $\delta V$ consists of a union of edges $\{k_{1},\ldots,k_{t}\}$, arranged in a convex polygon, belonging to distinct lines. By the definition of lines in general position, the edge $e_{j}$ passes through the interior of two of these edges, say $k_{1}$ and $k_{l}$, where $1<l\leq t$, and forms two new pairs of edges $\{k_{11},k_{12}\}$ and $\{k_{l1},k_{l2}\}$ in the graph $G_{n+1}$. We let $V_{1}$ be the region, bounded consecutively by the edges $\{k_{11},\ldots,k_{l-1},k_{l1},e_{j}\}$, and $V_{2}$ the region, bounded consecutively by the edges $\{k_{l2},k_{l+1},\ldots,k_{t},k_{12},e_{j}\}$. One of the regions is bounded, by the triangle with edges $\{e_{j},l_{1},l_{l}\}$, where $l_{1}$ and $l_{l}$ are the lines containing the edges $k_{1}$ and $k_{l}$ respectively. It follows, that both regions are convex, and, therefore, alcoves, by Lemma 0.5. Hence, $(*)$ is shown. Now, if $l_{n+1}$ is a new line, each of the $(n-1)$ new edges, either passes through the interior of an alcove in $G_{n}$, in which case, by the above $(*)(**)$, and Lemma 0.11, one extra alcove is introduced into the graph $G_{n+1}$, or, does not pass, through an interior, in which case, by the above $(*)(**)$, and Lemma 0.9, an extra alcove is also introduced into the graph $G_{n+1}$. In total, $(n-1)$ new alcoves are introduced, which implies that the total number of alcoves in $G_{n+1}$ is;\\

${(n-1)(n-2)\over 2}+(n-1)={n(n-1)\over 2}={([n+1]-1)([n+1]-2)\over 2}$\\

This implies the result, by induction.

\end{proof}

\begin{rmk}
\label{infinity}
We return to the notation of Lemma \ref{equal}, and the following remark. For $n$ odd, we can apply the previous lemma, to obtain that there exist $C^{n-1}_{2}$ alcoves associated to a regular bounded arrangement of lines. One may also extend the above definition of an alcove to regions in the real projective plane ${\mathcal R}P^{2}$, by, for example, assuming that all the intersections are in finite position. This may always be achieved by an appropriate choice of the line at $\infty$, so as not to include any of the vertices. With the convention that any two such regions intersecting in a vertex, on the line at $\infty$, are counted as a single alcove, the reader is invited to check that there are again $C^{n-1}_{2}$ alcoves, associated to a set of lines in general position. The reason for this
\end{rmk}

We can give a convenient description of the alcoves associated to a regular bounded line arrangement;\\

\begin{lemma}
\label{positions}
In the situation of Lemma \ref{equal}, and Lemma \ref{intersections}, the alcoves are defined by;\\

(i). For $n\geq 3$, the central alcove, with boundary defined by the $n$-polygon, inscribed in the unit circle.\\

(ii). For $n\geq 5$, $n$ peripheral alcoves of the first kind, inscribed between the unit and first concentric circle, with boundaries defined by the triangles, formed by the lines $\{l_{i},l_{i+1},l_{i+2}\}$, mod $(n)$.\\

(iii). For $n\geq 7$, $n$ peripheral alcoves of the second kind, inscribed between the
$(j-1,j,j+1)$ concentric circles, with boundaries defined by the quadrilaterals, formed by the lines\\

$\{l_{i},l_{i+1},l_{i+2j},l_{i+2j+1}\}$, mod $(n)$, $1\leq j\leq ({n-5\over 2})$\\

\end{lemma}

\begin{proof}
The proof is left to the reader, one should observe that the total number of alcoves is correct, as;\\

$1+n+n.{(n-5)\over 2}=1+n.{(n-3)\over 2}={(n-1)(n-2)\over 2}$

\end{proof}

\begin{rmk}
\label{class}
Observe that, for a nonsingular plane projective curve $C\subset P^{2}(\mathcal{C})$ of degree $n$, if $m$ is the class of $C$, then;\\

$m=n(n-1)=2n+2n+2({n(n-5)\over 2})$ $(*)$\\

Under certain further constraints on $C$, we can construct a $1$-parameter family $\{C_{t}:t\in Par_{t}\}$, with $C_{0}=C$, and $C_{\infty}$ consisting of $n$ lines $\{l_{1},\ldots,l_{n}\}$ in general position, with intersections described by the configurations in Lemmas 0.1 and Lemma 0.3, such that for each of the intersections $l_{i}\cap l_{j}$, $1\leq i<j\leq n$, there exist \emph{exactly} $2$ vertical tangents specialising to $l_{i}\cap l_{j}$, $(**)$. Using $(*)$ and Lemma 0.14, this suggests that the class points are uniformly distributed in three parts, across the periphery of the central alcove, the $n$ peripheral alcoves of the first kind, and the ${n(n-5)\over 2}$ peripheral alcoves of the second kind. The proof of $(**)$ will be the subject of the next section.

\end{rmk}
\end{section}

\begin{section}{Harmonic Variations}

\begin{rmk}
\label{pfs}
We observe some consequences of the degree-genus formula, Theorem 3.36 of \cite{dep1}, assuming Severi's conjecture, (\footnote{With the extra condition that the degeneration is asymptotic}, see \cite{dep2}), that, for any plane projective algebraic curve $C$, of degree $n$, having at most nodes as singularities, there exists an asymptotic family, see \cite{dep2}, $\{C_{t}:t\in P^{1}\}$, with the property that $C_{0}=C$ and $C_{\infty}$ is a union of $n$ lines in general position.
\end{rmk}

\begin{defn}
\label{harmonic1}

Let $\{l_{1},\ldots,l_{n}\}\subset P^{2}(\mathcal{R})$ be a sequence of $n$ projective lines, with coordinates $(x,y)$. We say that $\{l_{1},\ldots,l_{n}\}$ forms a harmonic arrangement if they satisfy the conditions of Lemma \ref{equal}, in the case that $n$ is odd, and, if, the intersections are in finite position, and satisfy the conditions of Lemma \ref{equal}, after a linear change of variables.

\end{defn}

\begin{defn}
\label{harmonic2}
Let $\{l_{1},\ldots,l_{n}\}\subset P^{2}(\mathcal{C})$ be a sequence of $n$ projective lines, defined over $\mathcal{R}$ and let $i:P^{2}(\mathcal{R})\rightarrow P^{2}(\mathcal{C})$ be the canonical inclusion. We say that $\{l_{1},\ldots,l_{n}\}$, forms a harmonic arrangement, if the pullbacks $\{i^{*}(l_{1}),\ldots,i^{*}(l_{n})\}$ form a harmonic arrangement in the sense of Definition \ref{harmonic1}.
\end{defn}

\begin{defn}
\label{harmonic3}
Let $C$ be a nonsingular plane projective curve of degree $n$. We say that $C$ is harmonic if there exist $n$ lines, $\{l_{1},\ldots,l_{n}\}$, which are bitangent to $C$, (\footnote{In the sense that there exist exactly $2$ points $\{p_{1,j},p_{2,j}\}$, on each $l_{j}$, such that $I_{p_{i,j}}(C,l_{j})=2$, for $1\leq i\leq 2$, and no further points of higher multiplicity}), which form a harmonic arrangement, in the sense of Definition \ref{harmonic2}, and such that, there exists lines $l_{a}$ and $l_{b}$, with $(l_{a}\cap C)=\{p_{1,j}:1\leq j\leq n\}$, and $(l_{b}\cap C)=\{p_{2,j}:1\leq j\leq n\}$.

\end{defn}

\begin{defn}
\label{harmonic4}
Let $C$ be a harmonic curve of degree $n$, in the sense of Definition \label{harmonic3}. Let $\{C_{t}:t\in Par_{t}\}$ be a $1$-dimensional family of  nonsingular plane projective curves, (\footnote{In the sense that $Par_{t}\subset P^{{(n+1)(n+2)\over 2}}$ is a $1$-dimensional irreducible algebraic variety, containing the nonsingular curve $C$}). We say that the family is a harmonic variation, if, there exist $\{0,\infty\}\subset Par_{t}$, $C_{0}=C$, $C_{\infty}$ is a union of lines $\{l_{1},\ldots,l_{n}\}$, forming a harmonic arrangement, and $Par_{t}\subset W^{4n}$, where, $W^{4n}$ is defined as;\\

$\{C_{\bar a}:\noindent \bigwedge_{1\leq j\leq n}C_{\bar a}(p_{1,j}),C_{\bar a}(p_{2,j}),I_{p_{1,j}}(C_{\bar a},l_{j})=2,I_{p_{2,j}}(C_{\bar a},l_{j})=2 \}$\\
\end{defn}

\begin{rmk}
\label{harmonic5}
For a given harmonic variation, we can choose a coordinate system $(x',y')$ such that the lines $\{l_{a},l_{b}\}$ correspond to $\{x=0,x=1\}$, the intersection $(l_{a}\cap l_{b})=[0:1:0]$, and the intersections $l_{i}\cap l_{j}$ are in finite position, for $1\leq i<j\leq n$. We can keep track of the original configuration of lines, in $(x,y)$, from Definitions \ref{harmonic2}, \ref{harmonic3}, through a linear isomorphism $L:P^{1}(\mathcal{C})\rightarrow P^{1}(\mathcal{C})$.
\end{rmk}

\begin{lemma}
\label{harmonic6}
For any given plane nonsingular curve $C$ of degree $n$, there exists a finite sequence $\{C_{i}:0\leq i\leq r\}$ of nonsingular plane curves of degree $n$, linear systems $\{L_{i,i+1}\subset P^{{(n+1)(n+2)\over 2}}:0\leq i\leq r-1\}$, and parameters $\{a_{0},a_{h}\}\cup\{a_{1,i},a_{2,i}:1\leq i\leq r-1\}$, with $C_{i}\sim a_{1,i}$, in $L_{i-1,i}$, and $C_{i}\sim a_{2,i}$, in $L_{i,i+1}$, $C\sim a_{0}$ in $L_{0,1}$, $C\sim a_{h}$ in $L_{r-1,r}$, such that $C_{r}$ is harmonic.
\end{lemma}

\begin{proof}
For a $1$-dimensional (generically nonsingular) family of curves, let $V_{k}\subset Par_{t}\times P^{2}$ be defined by;\\

$V_{k}=\overline{\{(t,l):\exists_{\geq k,x}x\in (l\cap C_{t})\wedge I_{x}(l,C_{t})\geq 2\}}$\\

We have that for $k\geq 2$, $V_{k+1}\subseteq V_{k}$, and, by $(***)$ in footnote \ref{complex}, each $V_{k}$ is a finite cover of $Par_{t}$, of degree at most ${n(n-2)\over 2}$. Now, given $C$, a nonsingular curve of degree $n$, we use the following method to reduce the $k$-tangents, for $k\geq 3$, to bitangents. $(****)$, Enumerate the $k$-tangent lines, for $k\geq 3$, as $\{l_{k,1},\ldots,l_{k,s(k)}\}$, and the bitangents as $\{l'_{1},\ldots,l'_{r}\}$. Pick $3$ points $\{p_{1},p_{2},p_{3}\}$ on $l_{3,1}$, centred at $\{(a,b),(a',b'),(a'',b'')\}$, with tangent lines $l_{p_{1}}=l_{p_{2}}=l_{p_{3}}=l_{1}$, defined by $cx+dy-(ca+db)=0$. Let $W^{4},W^{6}\subset P^{{(n+1)(n+2)\over 2}}$ be the hyperplanes, defined by;\\

\noindent $W^{4}=\{\overline{a}:g(\overline{a},a,b)=g(\overline{a},a',b')=0,d{\partial g\over\partial x}|_{(\overline{a},a,b)}-c{\partial g\over\partial y}|_{(\overline{a},a,b)}$\\

\indent \ \ $=0, d{\partial g\over\partial x}|_{(\overline{a},a',b')}-c{\partial g\over\partial y}|_{(\overline{a},a',b')}=0 \}$\\

\noindent $W^{6}=\{\overline{a}:g(\overline{a},a,b)=g(\overline{a},a',b')=g(\overline{a},a'',b'')=0,d{\partial g\over\partial x}|_{(\overline{a},a,b)}-c{\partial g\over\partial y}|_{(\overline{a},a,b)}$\\

\indent \ \ $=0, d{\partial g\over\partial x}|_{(\overline{a},a',b')}-c{\partial g\over\partial y}|_{(\overline{a},a',b')}=0, d{\partial g\over\partial x}|_{(\overline{a},a'',b'')}-c{\partial g\over\partial y}|_{(\overline{a},a'',b'')}=0\}$\\

where $g(\overline{a},x,y)=\sum_{i+j\leq n}a_{ij}x^{i}y^{j}$. We have that $codim(W^{4})=4$, $codim(W^{6})=6$, $W^{6}\subset W^{4}$, and  $\overline{a_{0}}\in W^{6}$, where $C=C_{\overline{a_{0}}}$. Choose a line $l''\subset W^{4}$, with $l''\cap W^{6}=\overline{a_{0}}$. Then $l''$ defines a $1$-parameter family of (generically nonsingular) curves $\{C_{t}:t\in l''\}$, of degree $n$.\\

Let $m$ denote the degree of the cover $V_{2}/l''$, and let $U\subset l''$ have the property that $Card(V_{2}(t))=m$, for $t\in U$. Let $D\subset U\times P^{1}\times P^{2}$ denote the family of curves defined by $D_{(t,s)}=\prod_{(l,t)\in V^{2}(t)}(y+l_{1}x+(l_{2}+s)$, with corresponding closure $\overline{D}\subset l''\times P^{1}\times P^{2}$. Let $W\subset l''\times P^{1}\times P^{2}$ be defined by;\\

$W(t,s,p)\equiv p\in C_{t}\cap \overline{D}(t,s)$\\

Then, using Bezout's theorem, $W$ defines a finite cover of $l''\times P^{1}$ of degree $mn$.

Using factoring multiplicity, see Lemma 2.6 of \cite{dep5}, for $(t_{0},0,p_{0})\in W$, we have that;\\

$Mult_{(t_{0},0,p_{0})}(W/(l''\times P^{1}))$\\

$=\sum_{s\in\mathcal{V}_{0}\ generic,p_{0,i}\in (W(t_{0},s)\cap\mathcal{V}_{p_{0}})}Mult_{(t_{0},s,p_{0,i})}(W(s)/l'')$ $(\dag)$\\

$=\sum_{t\in\mathcal{V}_{t_{0}}\ generic,q_{0,j}\in (W(t,0)\cap\mathcal{V}_{p_{0}})}Mult_{(t,0,q_{0,j})}(W(t)/P^{1})$ $(\dag\dag)$\\

If $p_{0}\in l_{0}$, we have that $(\dag)=I_{p_{0}}(C_{t_{0}},l_{0})Mult_{(t_{0},l_{0})}(V_{2}/l'')$\\

and $(\dag\dag)=Mult_{(t_{0},l_{0})}(V_{2}/l'')(\sum_{q_{0,j}\in W(t,0)}I_{(q_{0,j},l_{0,j})}(C_{t},l_{0,j}))$\\

where $q_{0,j}\in l_{0,j}$. Hence;\\

$I_{p_{0}}(C_{t_{0}},l_{0})=\sum_{q_{0,j}\in W(t,0)}I_{(q_{0,j},l_{0,j})}(C_{t},l_{0,j})$ $(\dag\dag\dag\dag)$\\

It follows that, for $\overline{a}_{0}\in l''$, we have that the total multiplicity;\\

$K=\sum_{l\in V_{2}(\overline{a}_{0}), p\in (l\cap C_{\overline{a}_{0}})}I_{p}(C_{\overline{a}_{0}},l)=\sum_{l\in V_{2}(\overline{a}),p\in (l\cap C_{\overline{a}})}I_{p}(C_{\overline{a}},l)$\\

for generic $\overline{a}\in (\mathcal{V}_{\overline{a}_{0}}\cap l'')$, hence, for all $\overline{a}\in \{{(\mathcal{V}_{\overline{a}_{0}}\cap l'')\setminus{\overline{a}_{0}}}\}$.\\

Removing the points of contact $1$, we obtain, for $l\in V_{2}(\overline{a}_{0})$, $p\in (l\cap C_{\overline{a}_{0}})$, with $I_{p}(C_{\overline{a}_{0}},l)\geq 2$, that;\\

$I_{p}(C_{\overline{a}_{0}},l)=\sum_{q\in C_{\overline{a}}\cap l'\cap\mathcal{V}_{p},l'\in V_{2}(\overline{a})\cap \mathcal{V}_{l}}I_{q}(C_{\overline{a}},l')$\\

$\geq \sum_{q\in C_{\overline{a}}\cap l'\cap\mathcal{V}_{p},l'\in V_{2}(\overline{a})\cap \mathcal{V}_{l},I_{q}(C_{\overline{a}},l')\geq 2}I_{q}(C_{\overline{a}},l')$\\

$L=\sum_{l\in V_{2}(\overline{a}_{0}), p\in (l\cap C_{\overline{a}_{0}}),I_{p}(C_{\overline{a}_{0}},l)\geq 2}I_{p}(C_{\overline{a}_{0}},l)$\\

$\geq \sum_{l\in V_{2}(\overline{a}),p\in (l\cap C_{\overline{a}}),I_{p}(C_{\overline{a}_{0}},l)\geq 2}I_{p}(C_{\overline{a}},l)$ $(\dag\dag\dag)$\\

(\footnote{\label{move1} It follows that, if $V(\overline{a}_{0},l)$ holds, and $p\in (C_{\overline{a}_{0}}\cap l)$, with $I_{p}(C_{\overline{a}_{0}},l)=w\geq 2$, and $Mult_{(\overline{a}_{0},l)}(V_{2}/l'')=b$, then, if $\epsilon>0$ is standard, with $B(p,\epsilon)^{c}\cap (C_{\overline{a}_{0}}\cap l)=B(l,\epsilon)^{c}\cap (V_{2}({\overline{a}_{0}}))=\emptyset$, the statement $\forall t'\in ({\mathcal{V}_{\overline{a}_{0}}\setminus \{\overline{a_{0}}\}}\cap l'')Q(t')$ holds, where;\\

$Q(t')\equiv\exists_{\bigwedge_{1\leq j\leq b}l_{j}}\exists_{\bigwedge_{1\leq j\leq b,1\leq a(j)\leq t(j)}p_{j,a(j)}}[l_{j}\in B(l,\epsilon)\wedge (t',l_{j})\in V_{2}(t')\wedge(p_{j,a_{j}}\in C_{t'}\cap l_{j}\cap B(p,\epsilon))\wedge I_{p_{j,a_{j}}}(C_{t'},l_{j})\geq 2\wedge\bigvee_{1\leq j\leq r,1\leq a(j)\leq t(j),\sum \theta_{(j,a_{j})}\leq w}$\\
$I_{p_{j,a_{j}}}(C_{t'},l_{j})=\theta_{(j,a_{j})}]$\\

where;\\

$I_{p_{j,a_{j}}}(C_{t'},l_{j})\geq 2\equiv \forall s\in {B(0,\epsilon)\setminus\{0\}}\exists_{w_{1}\neq w_{2}}\bigwedge_{1\leq i\leq 2}(w_{i}\in C_{t'}\cap (y+l_{1,j}x+(l_{2,j}+s))\cap B(p_{j,a_{j}},\epsilon))$\\

$I_{p_{j,a_{j}}}(C_{t'},l_{j})=\theta_{(j,a_{j})}\equiv \forall s\in {B(0,\epsilon)\setminus\{0\}}\exists^{=\theta_{j,a_{j}}}{w_{k}}\bigwedge_{1\leq k\leq \theta_{j,a_{j}}}(w_{k}\in C_{t'}\cap (y+l_{1,j}x+(l_{2,j}+s))\cap B(p_{j,a_{j}},\epsilon))$\\

Using Theorem 17.1 of \cite{dep4}, that a monad $\mu(p)$ coincides with an infinitesimal neighborhood $\mathcal{V}_{p}$, for $p\in P^{k}(\mathcal{C})$, and the fact that, for any infinite $n\in{^{*}\mathcal{N}}$, $B(p,{1\over n})\subset {^{*}U}$, for any open set $U$ in the complex topology, the property $Q$ holds for all infinite $n\in{^{*}\mathcal{N}}$, with $t'\in {B({\overline{a}_{0}},{1\over n})\setminus\{\overline{a}_{0}\}\cap l''}$. By the underflow principle and transfer, see \cite{cut}, it holds in the standard model, for all $n\in\mathcal{N}$, $n\geq k$, for some $k\in\mathcal{N}$, with $t'\in {B({\overline{a}_{0}},{1\over n})\setminus\{\overline{a}_{0}\}\cap l''}$, in particular, for $t'\in {B({\overline{a}_{0}},{1\over k})\setminus\{\overline{a}_{0}\}\cap l''}$.\\

Repeating this argument, for each $p\in (C_{\overline{a}}\cap l)$, with $I_{p}(C_{\overline{a}},l)\geq 2$, and each bitangent line $l$, with corresponding $B({\overline{a}_{0}},{1\over k_{p,l}})$, it follows that, taking the intersection $\bigcap_{l,p}B({\overline{a}_{0}},{1\over k_{p,l}})$, the total multiplicity of the new bitangent points is lowered. We can then, wlog, move the initial curve $C_{0}$ to a point $\overline{b_{0}}\in l''$, for which the fibre $V_{2}(\overline{b_{0}})$ is (in the sense of Zariski structures) unramified.}).\\

By footnote \ref{move1}, we can assume that the fibre $V_{2}(\overline{a_{0}})$ is unramified in the sense of Zariski structures, $(****)$. As $V_{k+1}\subset V_{k}$ is relatively closed, for $k\geq 2$, we have that, for $\overline{a_{0}}'\in(\mathcal{V}_{\overline{a_{0}}}\cap l'')$, $V_{2}(\overline{a_{0}}')\cap\mathcal{V}_{l'_{j}}\subset V_{2}$, for $1\leq j\leq r$, $V_{2}(\overline{a_{0}}')\cap\mathcal{V}_{l_{k,j(k)}}\subset (V_{k+1})^{c}$, for $1\leq j(k)\leq s(k)$, $k\geq 4$, and $2\leq j(k)\leq s(k)$, $k=3$, $V_{2}(\overline{a_{0}}')\cap({\mathcal{V}_{l_{3,1}}})\subset V_{3}^{c}$, using the fact that $(({l''\setminus\{\overline{a_{0}}\}})\cap W^{6})=\emptyset$, which gives that $\overline{a_{0}}$ is not a base point of the $g_{n}^{1}$ defined by the $l''$, (intersecting with $l_{3,1}$), and Lemma 2.10 of \cite{dep3}. It follows again that the statement $\forall t'\in ({\mathcal{V}_{\overline{a}_{0}}\setminus \{\overline{a_{0}}\}}\cap l'')P(t')$ holds, where;\\

$P(t')\equiv\exists_{(\geq r+1, l_{j}')}\exists_{(\geq s(3)-1,l_{3,j'},l_{3,j'}\neq l_{j}' )}\ldots\exists_{(\geq s(k),l_{k,j_{k}},l_{k,j_{k}}\neq l_{j}'\cup_{3\leq s\leq k-1}l_{s,j_{s}})}$\\

$[(l'_{j}\in V_{2}(t'))\wedge l_{3,j'}\in V_{4}^{c}(t')\wedge\ldots\wedge l_{k,j_{k}}\in V_{k+1}^{c}(t')]$\\

Using the argument of footnote \ref{move1} again, it follows that the property $P$ holds for $t'\in {B({\overline{a}_{0}},{1\over k})\setminus\{\overline{a}_{0}\}\cap l''}$, for some $k\in\mathcal{N}$. Choosing ${\overline{a}_{1}}\in {B({\overline{a}_{0}},{1\over k})\setminus\{\overline{a}_{0}\}\cap l''}$, and, using the result of footnote \ref{move1}, it follows that, for the new curve $C_{\overline{a}_{1}}$, the total weight $\sum_{k\geq 3}s(k)$ is \emph{strictly} (compare $(\dag\dag\dag)$) reduced, $(*****)$. Now repeating the argument from $(****)$, and using $(*****)$, we obtain, after a finite number $c$ of steps, a nonsingular curve $C_{\overline{a}_{c}}$, with the property that it has no $k$-tangents for $k\geq 3$, and exactly ${n(n-2)\over 2}$ bitangents, $(******)$.\\

Relabelling $C_{\overline{a}_{c}}$ as $C_{\overline{a}_{0}}$, we choose $n$ bitangent lines $\{l_{1},\ldots,l_{n}\}$. We now show how to obtain the condition that the lines intersect in exactly ${n(n-1)\over 2}$ points. $(\sharp)$, Suppose not, then, wlog, $\{l_{1},l_{2},l_{3}\}$ intersect in a point $q$. Let $\{p_{1,1},p_{1,2},p_{2,1},p_{2,2},p_{3,1},p_{3,2}\}$ denote the $6$ distinct tangent points on $\{l_{1},l_{2},l_{3}\}$. Let $\{W^{10},W^{12}\}\subset P^{(n+1)(n+2)\over 2}$ be defined as above, with $W^{10}$ defining curves of degree $n$, bitangent to $\{l_{1},l_{2}\}$ at $\{p_{1,1},p_{1,2},p_{2,1},p_{2,2}\}$, and $W^{12}$ defining curves of degree $n$, bitangent to $\{l_{1},l_{2}\}$ at $\{p_{1,1},p_{1,2},p_{2,1},p_{2,2}\}$, and tangent to $l_{3}$ at $p_{3,1}$. Again, choose a line $l'''\subset W^{10}$, with $l'''\cap W^{12}={\overline{a}_{0}}$. Using $(\dag\dag\dag\dag)$, the points $\{l_{1},\ldots,l_{n}\}$ in the fibre $V_{2}(\overline{a}_{0})$ are unramified, (\footnote{If one of the lines $l_{1}$ ramifies to $\{l_{1}',l_{1}''\}$, considering the cover $W$, and using $(\dag\dag\dag\dag)$,  we obtain $4$ points $p_{1,i,j}\in l_{i}'\cap\mathcal{V}_{p_{1,i}}$, $1\leq i,j\leq 2$, with $I_{p_{1,i,j}}(C_{\overline{a}_{0}}',l_{i})=1$, and no further points $p\in l_{i}'$, with $I_{p}(C_{\overline{a}_{0}}',l_{i})\geq 2$. It follows that the lines $\{l_{1}',l_{1}''\}$ can no longer even be tangent to the curve, let alone bitangent.}) Then, if $\overline{a}'_{0}\in({\mathcal{V}_{\overline{a}_{0}}\setminus{\overline{a}_{0}}})$, the corresponding $\{l'_{1},l'_{2},l'_{3}\}$ intersect in $3$ distinct new points $\{q,q_{1},q_{2}\}\subset\mathcal{V}_{q}$. Again, using the argument in footnote \ref{move1}, and, considering the cover $T/l'''$, defined by $T(p,t)\equiv \exists (l_{1},l_{2})(l_{1}\neq l_{2}\wedge p\in l_{1}\cap l_{2}\wedge \bigwedge_{1\leq i\leq 2}V_{2}(l_{i},t)$, it follows that the total number of intersections between the $n$ bitangent lines is increased, in $({\mathcal{V}_{\overline{a}_{0}}\setminus{\overline{a}_{0}}})$. Again, using the argument of footnote \ref{move1}, this property holds on some ${B({\overline{a}_{0}},{1\over k})\setminus{\overline{a}_{0}}}$, $k\in\mathcal{N}$. Now, we can choose $\overline{a}_{1}\in{B({\overline{a}_{0}},{1\over k})\setminus{\overline{a}_{0}}}$ and obtain a new curve $C_{\overline{a}_{1}}$, with this property. Again, using $(\dag\dag\dag)$, and the argument of footnote \ref{move1}, with $B({\overline{a}_{0}},{1\over k'})\subset B({\overline{a}_{0}},{1\over k})$, $k'>k$,  the condition $(******)$ is maintained. Repeating the argument, from $(\sharp)$, we obtain, after a finite number $c'$ of steps, a curve $C_{\overline{a}_{c}}$, with $n$ bitangent lines $\{l_{1},\ldots,l_{n}\}$, intersecting in ${n(n-1)\over 2}$ points, $(\sharp\sharp)$.\\

Now, again relabelling $C_{\overline{a}_{c}}$ to $C_{\overline{a}_{0}}$, with bitangent points\\
 $B=\{p_{1,1},\ldots,p_{1,n},p_{2,1},\ldots,p_{2,n}\}$, (wlog in finite position) we show how to preserve the condition $(\sharp\sharp)$ and find lines $\{l_{a},l_{b}\}$, with $\{p_{1,1},\ldots,p_{1,n}\}\subset l_{a}$ and $\{p_{2,1},\ldots,p_{2,n}\}\subset l_{b}$, $(!!!)$. Choose $l_{a}\neq l_{1}$, passing through $p_{1,1}$, intersecting $\{l_{2},l_{3},\ldots,l_{n}\}$ at the distinct points $\{q_{2},\ldots,q_{n}\}$ in finite position, distinct from $\{p_{1,2},\ldots,p_{1,n},p_{2,2},\ldots,p_{2,n}\}$, any of the other transverse intersections between $C_{\overline{a}_{0}}$ and $\{l_{2},l_{3},\ldots,l_{n}\}$, and the intersections $\{p_{i,j}=(l_{i}\cap l_{j}):1\leq i<j\leq n\}$, $(\sharp\sharp\sharp)$. We follow the argument in the following footnote \ref{variation}. After $n-1$ steps, we obtain a curve $C_{\overline{a}_{n-1}}$, such that the new tangent points to  $\{q_{1}=p_{1,1},q_{2},\ldots,q_{n}\}$ to $\{l_{1},l_{2},\ldots,l_{n}\}$ intersect the line $l_{a}$ transversely, and the bitangents $\{l_{1},\ldots,l_{n}\}$, formed by $\{q_{1},p_{2,1},\ldots,q_{n},p_{2,n}\}$ are in general position, $(!!)$. Now choose $l_{b}$, passing through $p_{2,1}$, such that the intersections with $\{l_{2},l_{3},\ldots,l_{n}\}$ at the distinct points $\{r_{2},\ldots,r_{n}\}$ are in finite position, distinct from $\{q_{1},\ldots,q_{n},p_{2,2},\ldots,p_{2,n}\}$, and any of the other transverse intersections between $C_{\overline{a}_{0}}$ and $\{l_{2},l_{3},\ldots,l_{n}\}$, and the intersections $\{p_{i,j}=l_{i}\cap l_{j}:1\leq i<j\leq n\}$ . Again, using the argument in footnote \ref{variation}, and, repeating the $(n-1)$ steps from $(!!)$, replacing $\{q_{2},\ldots,q_{n}\}$ with $\{r_{2},\ldots,r_{n}\}$, $(r_{1}=p_{2,1}$, we obtain a curve $C_{\overline{a}_{2(n-1)}}$, with the required property $(!!!)$, that the bitangent lines $\{l_{1},l_{3},\ldots,l_{n}\}$
are in general position, and the tangent points $\{q_{1},\ldots,q_{n},r_{1},\ldots,r_{n}\}$ lie on the lines $l_{a}$ and $l_{b}$ respectively, (\footnote{\label{variation} Moving tangents on fixed bitangent lines; for a given bitangent line $l_{j}$, with tangents $a,b$, and target $c$,  move $b$ to $c$, keeping $a$ fixed, while preserving bitangent conditions on the other lines $\{l_{i}:i\neq j\}$, with bitangents $\{p_{k,i}:1\leq k\leq 2,i\neq j\}$. Let the object curve $C$, of degree $n$, be denoted by $C_{\overline{c}}$, for $\overline{c}\in P^{{(n+1)(n+2)\over 2}}$.\\

Consider the irreducible dual curve $(C_{\overline{c}})^{*}$, with nodes $\{\nu_{j}:1\leq j\leq t\}$, $t={n(n-2)\over 2}$, and cusps $\{\kappa_{j}:1\leq j\leq s\}$, corresponding to the bitangents (the first $n$ nodes corresponding to the bitangent array considered above) and inflexions of $C_{\overline{c}}$, see Lemma \ref{harmonic} and Theorem 5.1 of \cite{dep1}. Assuming that the cusps are ordinary, that is of character $(2,1)$, we have that $n=deg(C_{\overline{c}})=cl((C_{\overline{c}})^{*})={3deg((C_{\overline{c}})^{*})-s\over 3}=n(n-1)-{s\over 3}$, using Theorem 6.4 of \cite{dep1}, (in particular $s=3n(n-2)$). We consider the Severi variety, $V_{d}^{t,s}=V_{n(n-1)}^{{n(n-2)\over 2},3n(n-2)}$, consisting of curves of degree $d=n(n-1)$, with $t={n(n-2)\over 2}$ nodes and $s=3n(n-2)$ cusps. Using deformation theory, developed in \cite{Ser}, $dim(V_{d}^{t,s})=3d+g-1-s$, where $g=g(C_{\overline{c}}^{*})=g(C_{\overline{c}})={(n-1)(n-2)\over 2}$, so $dim(V_{d}^{t,s})=3n(n-1)+{(n-1)(n-2)\over 2}-1-3n(n-2)={n(n+3)\over 2}$. We let $B_{n,d}\subset P^{{(d+1)(d+2)\over 2}}$ be the linear space of codimension $n$, consisting of curves of degree $d=n(n-1)$, passing through $\{\nu_{j}:1\leq j\leq n\}$, and $Z_{d}^{t,s}=(V_{d}^{t,s}\cap B_{n,d})$.  By presmoothness, we have that $dim_{comp}(Z_{d}^{t,s})\geq dim(V_{d}^{t,s})-n={n(n+3)\over 2}-n={n(n+1)\over 2}$. (As above, for $\overline{e}\in V_{d}^{t,s}$, we have that the tangent space $T_{\overline{e}}(V_{d}^{t,s})$, consists of curves of degree $d$ passing through the $t$ nodes and $s$ cusps of $C_{\overline{e}}$, (see Severi's calculations on curves infinamente vicine, in \cite{dep1}).)\\

We let $\Phi:V_{d}^{t,s}\rightarrow P^{(n+1)(n+2)\over 2}$ be the duality map, $C_{\Phi(\overline{e})}=(C_{\overline{e}})^{*}$, for $\overline{e}\in V_{d}^{t,s}$, and let $W_{d}^{t,s}\subset P^{(n+1)(n+2)\over 2}$, $W_{d}^{t,s}=Im(\Phi)\cong V_{d}^{t,s}$  be the corresponding variety of nonsingular curves of degree $n$, and $t={n(n-2)\over 2}$ bitangents, $s$ inflexions, and $Y_{d}^{t,s}\subset W_{d}^{t,s}$ be the corresponding variety to $Z_{d}^{t,s}$, $dim_{comp}(W_{d}^{t,s})\geq {n(n+1)\over 2}$. Let $G_{1,j,c}=(\overline{Y_{d}^{t,s}}\cap A_{1,j,c}$, then, again, by presmoothness, we have that $dim_{comp}(G_{1,j,c})\geq {n(n+1)\over 2}-(7n-4)={n^{2}-13n+8\over 2}$. Suppose there exists an $\overline{e}\in G_{1,j,c}$, such that the corresponding curve $C_{\overline{e}}$ is irreducible, $(\sharp\sharp)$. We claim that $Sing(C_{\overline{e}})\cap(\{p_{k,j}:1\leq k\leq 2,i\neq j\}\cup\{a,c\})=\emptyset$, $(***)$.  In order to see $(***)$, suppose that there exists a singularity of $C_{\overline{e}}$, centred at one of the bitangents $p_{1,i}$ or $p_{2,i}$, for $j\neq i$. As ${\overline{e}}\in Y_{d}^{t,s}$, we have either that, say Case 1, $I_{p_{1,i}}(C_{\overline{e}},l_{i})\geq 3$ (one of the nodal branches is tangent to $l_{i}$), or Case 2, there exists another tangent of $C_{\overline{e}}$, centred at $q\in({l_{i}\setminus \{p_{1,i},p_{2,i}\}})$; as the corresponding dual curve $(C_{\overline{e}})^{*}$, belonging to $\overline{Z_{d}^{t,s}}$, has either a cusp or node singularity at the point $\nu_{i}$, corresponding to the bitangent line $l_{i}$. Choose an irreducible curve $C_{i}\subset Y_{d}^{t,s}\cap A_{2,j,x}$, containing $\{\overline{c},\overline{e}\}$, with $C_{i}\cap A_{1,j,c}=\overline{c}$, so for $\overline{e}'\in({C_{i}\setminus \{\overline{e}\}})$, the corresponding curve $C_{\overline{e}'}$ is nonsingular. Consider the cover $R_{i}\subset C_{i}\times l_{i}$, defined by $R_{i}(\overline{t},x)\equiv x\in (l_{i}\cap C_{\overline{t}})\wedge \overline{t}\in C_{i}$. In Case 1, we claim there exists an irreducible component $X_{1}$ of $R_{i}$, passing through $(\overline{e},p_{1,i}$, such that $pr_{2}(X_{1})=l_{i}$, $pr_{1}(X_{1})=C_{i}$. This follows from the fact, that if $\overline{e}'\in({\mathcal{V}_{\overline{e}}\setminus \{\overline{e}\}})$, there exists $p'_{1,i}\in ({\mathcal{V}_{p_{1,i}}\setminus \{p_{1,i}\}})$, with $(\overline{e}',p'_{1,i})\in R_{i}$. In order to see this, choose a direction $(x_{0},y_{0})$, not on $l_{i}$, and, for $s\in P^{1}$, let $C^{s}_{\overline{t}}(x,y)=C_{\overline{t}}(x+sx_{0},y+sy_{0})$. Considering the cover $T_{i}\subset P^{1}\times C_{i}\times l_{i}$, defined by $T_{i}(s,\overline{t},x)\equiv x\in C^{s}_{\overline{t}}\cap l_{i}$, we have that $Mult_{(p_{1,i},0,\overline{e})}(T_{i}/P^{1}\times C_{i})\geq 3$, $Mult_{(p_{1,i},0,\overline{e}')}(T_{i}/P^{1}\times C_{i})=2$, hence, by summability of specialisation, there exists $p'_{1,i}\in ({\mathcal{V}_{p_{1,i}}\setminus \{p_{1,i}\}})$, with $Mult_{(p_{1,i}',0)}(T_{i}(\overline{e}')/P^{1})\geq 1$. Taking an irreducible component of $X_{2}$ of $R_{i}$ through $(\overline{e},q)$ with $q\neq p'_{1,i}$, we either have that $X_{1}=X_{2}$, in which case $deg(X_{2}/C_{i})\geq 2$, so there exists $(\overline{e}'',q')\in X_{1}$, with $\overline{e}''\in ({C_{i}\setminus \{\overline{e}\}})$ and $Mult_(\overline{e}'',q')(R_{i}/C_{i})\geq 2$. As $C_{\overline{e}''}$ is nonsingular, we have that $q'$ defines a tangent with $l_{i}$. It follows that $l_{i}$ is a tritangent to $C_{\overline{e}''}$ and the corresponding dual curve $(C_{\overline{e}''})^{*}$, has a triple node at $\nu_{i}$, contradicting the definition of $\{Y_{d}^{t,s},Z_{d}^{t,s}\}$. Otherwise, $X_{1}\neq X_{2}$, and, as we can assume now that $deg(X_{1}/C_{i})=1$, we can find again find an intersection $(\overline{e}'',q')\in (X_{1}\cap X_{2})$, with $Mult_(\overline{e}'',q')(R_{i}/C_{i})\geq 2$, $\overline{e}''\in ({C_{i}\setminus \{\overline{e}\}})$, $q\neq p'_{1,i}$, and we can use the same argument as before. (Case 2 is similar, and the line $l_{j}$). It follows that $(***)$ holds.\\

We let $B_{1,j,c}$ be the linear space of codimension $4(n-1)+4=4n$, consisting of curves of degree $n$, tangent to $l_{i}$, $i\neq j$, at remaining bitangents, tangent to $l_{j}$ at $a$ and $c$, and $B_{2,j,c}$, the linear space of codimension $4(n-1)+2=4n-2$, consisting of curves of degree $n$, tangent to $l_{i}$, $i\neq j$, at remaining bitangents, tangent to $l_{j}$ at $a$, so $B_{1,j,c}\subset B_{2,j,c}$. We have that the curve $C_{lines}\in B_{1,j,c}$, where $C_{lines}$ consists of the union $\bigcup_{1\leq j\leq n}l_{j}$. We let $B_{h_{i,j},1,j,c}\subset B_{1,j,c}\subset B_{2,j,c}$ be the ${n(n-1)\over 2}$ linear spaces of codimension $4n+1$, consisting of curves $C_{\overline{l}}$ in $B_{1,j,c}$, with $h_{i,j}\in C_{\overline{l}}$, where $h_{i,j}=(l_{i}\cap l_{j})$, for $i\neq j$. Choose $C_{\overline{a}}$, with $\overline{a}\in B_{1,j,c}$ generic, and let $l=span(\overline{a},lines)$, so $l\subset B_{1,j,c}$. If $C_{\overline{a}}$ is irreducible, then using the result of $(++)$, applied to the linear system $l$, if $p$ is a singularity of $C_{\overline{a}}$, then $p$ must be situated at an intersection point $h_{i,j}$ for some $i\neq j$. As $\overline{a}$ is generic, we have that $h_{i,j}\notin C_{\overline{a}}$, for $i\neq j$, hence $C_{\overline{a}}$ is nonsingular. If $C_{\overline{a}}$ is reducible, with, wlog, irreducible components $\{C_{1},C_{2}\}$, then, using $(!!),(!!!!!)$, we have that $C_{1}=\bigcup_{j\in J}l_{j}$, for some $J\subset \{1,\ldots,n\}$, with $Card(J)=n_{1}$, and there exists an isomorphic linear system $L_{2}$ of curves $D_{i_{2}^{-1}(\overline{l})}$, with degree $n-n_{1}$, $i_{2}:L\rightarrow L_{2}$, with $C_{\overline{l}}=(C_{1}\cup D_{i_{2}^{-1}(\overline{l})}$, for $\overline{1}\in L$, with fixed singularities $\{p_{1},\ldots,p_{r}\}$ on $D_{\overline{l}}$. As above, we can assume that $\{p_{1},\ldots,p_{r}\}\cap \{h_{i,j}:1\leq i<j\leq n\}=\emptyset$. Moreover, $\{p_{1},\ldots,p_{r}\}\cap({(C_{1}\cup C_{2})\setminus C_{lines}})=\emptyset$, as the singularities are fixed. We must, then have that $\{p_{1},\ldots,p_{r}\}\subset (C_{0}\cap C_{1})$, but $C_{0}\subset C_{lines}$, hence, $\{p_{1},\ldots,p_{r}\}$ define singularities of $C_{lines}$, as $\{p_{1},\ldots,p_{r}\}\subset C_{1}\cap D_{i_{2}^{-1}(\overline{1})}$, for $\overline{1}\in L$. This contradiction gives that $C_{\overline{a}}$ is irreducible, and, hence, by the previous part, nonsingular.\\

Considering again the variety $Y_{d}^{t,s}$, let $K\subset P^{{(n+1)(n+2)\over 2}}$ be defined by;\\

$K(\overline{l})\equiv \bigwedge_{i=1}^{n}\exists(x_{i}\neq y_{i})(l_{i}(x_{i})\wedge l_{i}(y_{i})\wedge l_{x_{i}}=l_{y_{i}}=l_{i})$\\

Then we have that $K(lines)$, and $dim(\overline{{K\setminus{(K\cap Y_{d}^{t,s})}}})<dim(\overline{K})$, $dim(\overline{{Y_{d}^{t,s}\setminus{(Y_{d}^{t,s}\cap K)}}})<dim(\overline{Y_{d}^{t,s}})$. It follows, as $Y_{d}^{t,s}$ is irreducible, that $\overline{Y_{d}^{t,s}}(lines)$. Now, choosing $\overline{e}\in {({\overline{Y_{d}^{t,s}}}\cap B_{1,j,c})\setminus \bigcup_{i\neq j}B_{h_{i,j},1,j,c}}$ (do this..) generic, we claim that, if $l=span(\overline{e},lines)$, then $l\subset(\overline{Y_{d}^{t,s}}\cap B_{1,j,c})$, $(+)$. Suppose that $C_{\overline{e}}$ is irreducible. By the above proof of $(***)$, the singularities are disjoint from the bitangent points $\{p_{k,i}:1\leq k\leq 2,i\neq j\}\cup\{a,c\}$. It follows that, as the bitangent points $\{p_{k,i}:1\leq k\leq 2,i\neq j\}\cup\{a,c\}\subset (C_{lines}\cap C_{\overline{e}})$, they belong to every $C_{\overline{l}}$, with $\overline{l}\in l$, and for $\overline{e}'\in (({\mathcal{V}_{\overline{e}}\setminus \{\overline{e}\}})\cap l)$, we have that, they define nonsingular points of $C_{\overline{e}'}$. As $\overline{e}'$ is generic, we have that $\overline{e}'\notin B_{p_{i,j},1,j,c}$, for $1\leq i<j\leq n$, hence, by the above analysis $C_{\overline{e}'}$ is nonsingular. Moreover, using $(***)$, no third tangent can occur along any of the $\{l_{j}:1\leq j\leq n\}$. It follows that the dual curve $(C_{\overline{e}'})^{*}\in Z_{d}^{t,s}$, hence, $C_{\overline{e}}\in Y_{d}^{t,s}$, giving the result $(+)$. If $C_{\overline{e}}$ is reducible, then, again by the above analysis $C_{\overline{e}}$ is irreducible, and we obtain the result. Taking ${\overline{e}'}\in (({\mathcal{V}_{\overline{e}}\setminus \{\overline{e}\}})\cap l)$, we obtain a nonsingular curve $C_{\overline{e}'}\in {\overline{Y_{d}^{t,s}}}\cap B_{1,j,c}$, which satisfies the required properties of the footnote.\\

(Fixed Singularities 1) Let $L$ be a (generically) irreducible, $1$-dimensional linear system of plane curves of degree $n$, defined by $f(x,y;t)$,  Let $F\subset L\times P^{2}$ be defined by;\\

$F(t,\overline{z})\equiv \overline{z}\in Sing(C_{t})\equiv f(\overline{z},t)=0\wedge {\partial f\over \partial x}|_{\overline{z}}=0\wedge {\partial f\over \partial y}|_{\overline{z}}=0$\\

We claim that the irreducible components of $F$ are of the form $L\times\{(t_{0},{\overline{z}_{0}})\}$ or $(t_{0},{\overline{z}_{0}})$, with $(t_{0},{\overline{z}_{0}})\in L\times P^{2}$, or $(\{(t_{0}\}\times C'_{t_{0}})$, with $C'_{t_{0}}$ defining a non-reduced component of $C_{t_{0}}$, $(++)$.

Clearly every irreducible component of $F$, has dimension at most $1$, as $deg({\partial f\over \partial x}<deg(f)$. As the family, defined by $L$, is generically irreducible, we can remove the finitely many parameters $\{t_{j}:1\leq j\leq s\}\subset L$ defining curves $C_{t_{j}}$, with reduced components. Suppose there exists an irreducible component $F_{0}$, of dimension $1$, with $pr_{P^{2}}(F_{0})$ defining an irreducible curve $D\subset P^{2}$, of degree $m$, with $F_{0}\nsubseteq (\{(t_{0}\}\times P^{2})$, so $pr_{L}(F_{0})=L$. The series $W_{t}(\overline{z})\equiv \overline{z}\in(D\cap C_{t})$ defines a $g^{1}_{nm}$ on $D$ with parameter space $U=({L\setminus\{t_{j}:1\leq j\leq s\}})$, as if there exists $t'\in U$, with $D\subset C_{t'}$, $C_{t'}$ would contain a reduced component. Using the fact that $pr_{L}(F_{0})=L$, we have, for generic $t\in U$, that there exists $p\in ({D\setminus Base(L)})\cap Sing(C_{t})$. We have that $I^{L}(p,D,C_{t})=1$, and, by definition of $F$, that $I(p,D,C_{t})\geq 2$, see notation in \cite{dep3}. Using the result of \cite{dep3}, Lemma 2.10, (with the slight modification, that we have removed finitely many points from $L$) we obtain a contradiction. Hence, $F_{0}$ is of the form $L\times\{(t_{0},{\overline{z}_{0}})\}$, with $(t_{0},{\overline{z}_{0}})\in L\times P^{2}$, giving the claim $(++)$.\\

(Fixed Singularities, $2$) Let $L\subset Sing(C)$ be a generically irreducible linear system, then there exists an open set $U\subset L$, such that for each $\overline{a}\in U$, $C_{\overline{a}}$ has exactly $r$ singularities, centred at $\{p_{1},\ldots,p_{r}\}$, $(\sharp)$. To see this, suppose that a generic curve has $r$ singularities, so the condition holds on an open set $U_{r}\subset L$. The condition $B_{r+1}$, that there exist at least $r+1$ singularities is closed, hence holds on $U_{r}^{c}$. Choose independent generic points $\{\overline{a}_{1},\overline{a}_{2}\}$ from $L$, then the line $l\subset L$, connecting ${\overline{a}_{1}}$ and ${\overline{a}_{2}}$, intersects $B_{r+1}$ in finitely many points, not including $\{\overline{a}_{1},\overline{a}_{2}\}$.
Suppose $Sing(C_{\overline{a}_{1}})\neq Sing(C_{\overline{a}_{2}})$, and consider the linear system defined by $l\subset L$. Using the result of $(***)$, we have, if $q$ is a singularity of $C_{\overline{a}_{1}}$, not of $C_{\overline{a}_{2}}$, then, for ${\overline{a}_{1}}'\in\mathcal{V}_{\overline{a}_{1}}$, we have that $Sing(C_{\overline{a}_{1}}')=Sing({\overline{a}_{1}\setminus q})$, contradicting the fact there are no curves in the family with $r-1$ singularities. It follows that $Sing(C_{\overline{a}_{1}})\neq Sing(C_{\overline{a}_{2}})$, and, as this condition is definable, the result follows.\\

$(!)$ Let $L$ be a $1$-dimensional linear system, consisting of curves of degree $n$, then if the generic curve $C_{\overline{a}}$ is reducible with irreducible components $\{C_{1},C_{2}\}$, of degrees $\{n_{1},n_{2}\}$, such that $n_{1}+n_{2}=n$, then every curve in the family is reducible with components of degrees $\{n_{1},n_{2}\}$.
To see this, let $V\subset L\times P^{n_{1}}\times P^{n_{2}}$ be defined by;\\

$V(\overline{l},\overline{t_{1}},\overline{t_{2}})\equiv C_{\overline{l}}=C_{\overline{t_{1}}}C_{\overline{t_{2}}}$\\

then $V$ is closed, and by completeness of closed projective varieties, so is the projection $W\subset L$;\\

$W=\exists\overline{t_{1}}\exists\overline{t_{2}}V$\\

As $\overline{a}$ is generic and $W(\overline{a})$, we have that $W=L$ as required.\\

$(!!)$ Let $L$ be a (generically) reducible, $1$-dimensional linear system of plane curves of degree $n$. Then, if $\overline{a}\in L$ is generic, with $C_{\overline{a}}$, having irreducible components $\{C_{1},C_{2}\}$, with degrees $\{n_{1},n_{2}\}$, such that $n_{1}+n_{2}=n$, then, either there exists a (generically) irreducible $1$-dimensional linear system $L_{1}$, of plane curves of degree $n_{1}$, and a linear isomorphism $i_{1}:L_{1}\rightarrow L$, such that, for $\overline{l_{1}}\in L_{1}$, $C_{i_{1}(\overline{l_{1}})}=C_{\overline{l_{1}}}.C_{2}$, or, there exists a (generically) irreducible $1$-dimensional linear system $L_{2}$, of plane curves of degree $n_{2}$, and a linear isomorphism $i_{2}:L_{2}\rightarrow L$, such that, for $\overline{l_{2}}\in L_{2}$, $C_{i_{2}(\overline{l_{2}})}=C_{1}.C_{\overline{l_{2}}}$.\\

In order to see this, as in $(***)$, we let $F\subset L\times P^{2}$ be defined by;\\

$F(t,\overline{z})\equiv \overline{z}\in Sing(C_{t})\equiv f(\overline{z},t)=0\wedge {\partial f\over \partial x}|_{\overline{z}}=0\wedge {\partial f\over \partial y}|_{\overline{z}}=0$\\

where $f$ defines $L$. We claim that, if $F_{0}$ is an irreducible component of $F$, such that $dim(pr_{P^{2}}(F_{0}))=1$, and $F_{0}\nsubseteq (\{t_{0}\}\times P^{2})$, then $pr_{P^{2}}(F_{0})=C_{1}$ or $pr_{P^{2}}F_{0}=C_{2}$, $(!!!)$. Suppose not, then, we have, that for generic $\overline{l}'\in L$, $C_{\overline{l}'}\cap pr_{P^{2}}F_{0}$ is finite, otherwise, $C_{\overline{l''}}\supset pr_{P^{2}}F_{0}$, for all $\overline{l''}\in L$, which is not the case, as $C_{1}$ and $C_{2}$ are irreducible. As in $(***)$, if $D=pr_{P^{2}}(F_{0})$, with degree $m$, the series $W_{t}(\overline{z})\equiv \overline{z}\in(D\cap C_{t})$, defines a $g^{1}_{nm}$ on $D$ with parameter space $V$, where $V=({L\setminus\{t\in L, C_{t}\supset D\}})$. As above, for generic $t\in V$, and using the fact that $pr_{L}(F_{0})=L$, we can find $p\in ({D\setminus Base(L)})\cap Sing(C_{t})$. We have that $I^{L}(p,D,C_{t})=1$, and, by definition of $F$, that $I(p,D,C_{t})\geq 2$, see notation in \cite{dep3}. Again we obtain a contradiction.\\

Suppose that $n_{1}\neq n_{2}$. Let $V_{1}\subset L\times P^{{(n_{1}+1)(n_{1}+2)\over 2}}$, $V_{2}\subset L\times P^{{(n_{2}+1)(n_{2}+2)\over 2}}$ be defined by;\\

$V_{1}(\overline{l},\overline{l_{1}})\equiv L(\overline{l})\wedge C_{\overline{l}}\supset C_{\overline{l}_{1}}$\\

$V_{2}(\overline{l},\overline{l_{2}})\equiv L(\overline{l})\wedge C_{\overline{l}}\supset C_{\overline{l}_{2}}$\\

We have that, for $1\leq j\leq 2$, $V_{j}$ consists of a unique $1$-dimensional irreducible component $V_{j,0}$ with $deg(V_{j,0}/L)=1$, together with finitely many points $\{p_{j,k}:1\leq k\leq t(j)\}$. For $1\leq j\leq 2$, we let $i_{j}:L\rightarrow V_{j,0}$ be the unique isomorphisms such that $pr_{1}\circ i_{j}=Id_{L}$. We let $V_{1,2}\subset L\times P^{2}$ be defined by;\\

$V_{1,2}(\overline{l},p)\equiv p\in C_{i_{1}(\overline{l})}\cap C_{i_{2}(\overline{l})}$\\

We have that $V_{1,2}$ is a closed generically finite cover of $L$. Let $Z$ be an irreducible component of $V_{1,2}$, not contained in $\{t_{0}\}\times P^{2}$, for some $t_{0}\in L$. By presmoothness, we have that $dim(Z)=1$. Suppose that $pr_{P^{2}}(Z)$ defines an irreducible curve $D\subset P^{2}$. If $D\notin\{C_{1},C_{2}\}$, then, as $Z$ defines an irreducible component of $F$, defined above, we obtain, by $(!!!)$, a contradiction. Hence, for any irreducible component $Z$ of $V_{1,2}$, not contained in $\{t_{0}\}\times P^{2}$, either Case 1; $pr_{P^{2}}(Z)=C_{i}$, for $i=1$ or $i=2$, or Case 2; $Z=L\times\{p_{0}\}$, for some $p_{0}\in C_{1}\cap C_{2}$. In Case 1, wlog $pr_{P^{2}}(Z)=C_{1}$. Suppose that $C_{1}$ is not an irreducible component of every $C_{\overline{l}}$, with $\overline{l}\in L$, then, as the condition $(\forall{l}\in L)C_{\overline{l}}\supset C_{1}$, $(!!!!)$,  fails, and this condition is closed, it follows there exist finitely many parameters $\{t'_{k}:1\leq k\leq t\}$ such that, if $\overline{l}\in W=({L\setminus\{t'_{k}:1\leq k\leq s\}})$, then $(C_{\overline{l}}\cap C_{1})$ is finite. We then obtain a $g^{1}_{n_{1}n}$ on $C_{1}$, with parameter space $W$, given by $N(z)\equiv z\in C_{1}\cap C_{\overline{l}}$, $\overline{l}\in W$. Choosing $p\in ({C_{1}\setminus Base(W)})\cap pr_{P^{2}}(Z\cap pr^{-1}_{L}(\overline{a}))$, for some generic $\overline{a}\in W$,  (if this fails then there exists an open $Q\subset W$, with $C_{1}\cap pr_{P^{2}}(Z\cap pr^{-1}_{L}(Q))=Base(W)$, which is not the case. We thus obtain $p\in C_{1}\cap Sing(C_{\overline{a}})$, as $p\in C_{i_{1}(\overline{a})}\cap C_{i_{2}(\overline{a})}$. We have that $I^{W}(p,C_{1},C_{\overline{a}})=1$, and that $I(p,C_{1},C_{t})\geq 2$, see notation in \cite{dep3}. Again we obtain a contradiction. However, we claim that Case 2 cannot always occur. We make the further assumption that, for \emph{any} $\{\overline{l},\overline{l'}\}\subset L$, we have that $C_{i_{1}(\overline{l})}\cap C_{i_{2}(\overline{l'})}$ is finite, $(!!!!!)$. (this is slightly stronger than the requirement that the condition $(!!!!)$ fails, we will weaken it later).

Fix $\overline{l_{0}}\in L$, and consider the $g^{1}_{n_{1}n_{2},\overline{l_{0}}}$ on $C_{\overline{l_{0}}}$, with parameter space $L$, obtained by intersecting $C_{\overline{l_{0}}}$ with $C_{i_{2}(\overline{l})}$, for $\overline{l}\in L$. We then claim that there exists a multiple point $p_{\overline{l_{0}}}$ for this $g^{1}_{n_{1}n_{2},\overline{l_{0}}}$. In order to see this, consider the variety $G_{\overline{l_{0}}}\subset L\times P^{2}$, defined by $G_{\overline{l_{0}}}(\overline{l},p)\equiv (f_{1}(p;\overline{l_{0}})=f_{2}(p;\overline{l})=0\wedge det({\partial f_{i}\over\partial x_{j}})_{1\leq i,j\leq 2}|_{p,\overline{l}}=0$. By presmoothness, and assuming that $det({\partial f_{i}\over\partial x_{j}})_{1\leq i,j\leq 2}|_{p,\overline{l}}\neq c$, for $c\neq 0$, $G_{\overline{l_{0}}}\neq\emptyset$, witnessed by $(\overline{l_{1,0}},p_{\overline{l_{0}}})$. Using the result of \cite{dep3}, Lemma 2.10, (which generalises easily to reducible curves), as above, either $p_{\overline{l_{0}}}\in Base(g^{1}_{n_{1}n_{2},\overline{l_{0}}})$, for this system, or it ramifies, that is $I_{p_{\overline{l_{0}}}}^{L}(C_{i_{1}(\overline{l_{0}})},C_{i_{2}(\overline{l_{1,0}})})\geq 2$. In the former case, we have that $\{\overline{l}\in L:I_{p_{\overline{l_{0}}}}(C_{i_{1}(\overline{l_{0}})},C_{i_{2}(\overline{l})})\geq k\}$ is definable and linear, hence, if $k_{1}$ is the minimum multiplicity of the $g^{1}_{n_{1}n_{2},\overline{l_{0}}}$ at $p_{\overline{l_{0}}}$, there exists $\overline{l_{1,0'}}\in L$, with $I_{p_{\overline{l_{0}}}}(C_{i_{1}(\overline{l_{0}})},C_{i_{2}(\overline{l_{1,0'}})})\geq k_{1}+1$, and for generic $\overline{l}\in L$, $I_{p_{\overline{l_{0}}}}(C_{i_{1}(\overline{l_{0}})},C_{i_{2}(\overline{l})})=k_{1}$, $(\dag\dag)$, and again we obtain ramification in $L$, that is $I_{p_{\overline{l_{0}}}}^{L}(C_{i_{1}(\overline{l_{0}})},C_{i_{2}(\overline{l_{1,0'}})})\geq 2$. Wlog we use the notation $\overline{l_{1,0}}$ for $\overline{l_{1,0'}}$. Now consider the variety $S\subset L\times L$, given by;\\

$S(\overline{l},\overline{l}')\equiv (\exists p)(p\in (C_{i_{l}(\overline{l})}\cap C_{i_{2}(\overline{l}')})\wedge I_{p}^{L,g^{1}_{n_{1}n_{2},\overline{l}}}(C_{i_{1}(\overline{l})},C_{i_{2}(\overline{l}')})\geq 2)$\\

By the above analysis, we obtain that $S$ is a closed finite cover of $L$, hence, intersecting with the diagonal $\Delta\subset (L\times L)$, we can find $(\overline{l_{1}},\overline{l_{1}})\in S$, and $p_{\overline{l_{1}}}\in P^{2}$ with  $I_{p_{\overline{l_{1}}}}^{L}(C_{i_{1}(\overline{l_{1}})},C_{i_{2}(\overline{l_{1}})})\geq 2$. It follows that, taking $\overline{l_{1}}'\in({\mathcal{V}_{\overline{l_{1}}}\setminus\{\overline{l_{1}}\}})$, we can find distinct point $\{p_{1,\overline{l_{1}}}, p_{2,\overline{l_{1}}}\}\subset (\mathcal{V}_{p_{\overline{l_{1}}}}\cap C_{i_{1}(\overline{l_{1}})}\cap C_{i_{2}(\overline{l_{1}}')})$, where $p_{\overline{l_{1}}}\in (C_{i_{1}(\overline{l_{1}})}\cap C_{i_{2}(\overline{l_{1}})})$. Now consider the cover $Y\subset L\times L\times P^{2}$, defined by $Y(\overline{l},\overline{l}',p)\equiv p\in (C_{i_{l}(\overline{l})}\cap C_{i_{2}(\overline{l}')})$, then, using summability of specialisation, we obtain that $Mult_{(\overline{l_{l}},\overline{l_{l}},p_{\overline{l_{1}}})}(Y/\Delta)\geq 2$, hence, there exist $2$ distinct irreducible components $\{Z_{l},Z_{2}\}$ of $V_{1,2}$, projecting onto $L$, passing through $(\overline{l_{1}},p_{\overline{l_{1}}})$. Clearly, such components cannot both be of the form required in Case 2. It follows that Case 1 holds, and we obtain the result, as required.\\

To complete the proof, with just the assumption that $(!!!!)$ fails, we have to allow for the possibility, that, for any given $C_{i_{1}(\overline{l_{0}})}$, there exist finitely many parameters $P=\{\overline{l_{0,j}}:1\leq j\leq t(\overline{l_{0,j}})\}$ such that $C_{i_{1}(\overline{l_{0}})}\cap C_{i_{2}(\overline{l_{0,j}})}$ contains a component of dimension $1$. In this case, we can remove the parameters $P$, setting $W=({L\setminus P})$, and obtain a $g_{n_{1}n_{2}}^{1}$ on $C_{i_{1}(\overline{l_{0}})}$, with parameter space $W$. We can then complete the $g_{n_{1}n_{2}}^{1}$ to a $g_{n_{1}n_{2},c}^{1}$, with parameter space $L$, by letting $F\subset (L\times C_{i_{1}(\overline{l_{0}})})$ be defined by $F_{\overline{l_{0}}}=\overline{H_{\overline{l_{0}}}}$, where $\overline{H_{\overline{l_{0}}}}\subset W\times C_{i_{1}(\overline{l_{0}})}$ is given by $\overline{H_{\overline{l_{0}}}}(\overline{l},p)\equiv (p\in C_{i_{1}(\overline{l_{0}})}\cap C_{i_{2}(\overline{l})})\wedge W(\overline{l})$, and letting the weighted set $B_{\overline{l_{0,j}}}=pr_{P^{2}}(F(\overline{l_{0,j}}))$, with weights $Mult_{a}(F_{\overline{l_{0}}}/L)$, for $a\in pr_{P^{2}}F(\overline{l_{0,j}})$. It is an easy exercise, left to the reader, to show that results above hold for this more abstract definition.\\

If $n_{1}=n_{2}$, we let $V_{3}\subset L\times P^{{(n_{1}+1)(n_{1}+2)\over 2}}$ be defined by;\\

$V_{3}(\overline{l},\overline{l_{3}})\equiv L(\overline{l})\wedge C_{\overline{l}}\supset C_{\overline{l_{3}}}$\\

Then, as $|V_{3}({a})|=2$, for generic $\overline{a}\in L$, $V_{3}$ has either two irreducible components $V_{j,0}$, for $1\leq j\leq 2$, with $deg(V_{j,0}/L)=1$, or a single irreducible component $R$, with $deg(R/L)=2$. In the first case, we repeat the argument above to obtain the result. In the second case, we let $M_{1,2}\subset L\times P^{2}$ be defined by $M_{1,2}=\overline{W_{1,2}}$, where;\\

$W_{1,2}(\overline{l},p)\equiv \exists(\overline{{l}_{1}}\overline{{l}_{2}})( R(\overline{l},\overline{{l}_{1}})\wedge R(\overline{l},\overline{{l}_{2}})\wedge (\overline{{l}_{1}}\neq \overline{{l}_{2}})\wedge (p\in C_{\overline{{l}_{1}}}\cap C_{\overline{{l}_{2}}})$\\

Arguing, as above, with $M_{1,2}$ replacing $V_{1,2}$, and observing that there exists an open set $U\subset L$, with $pr_{P^{2}}(M_{1,2}(\overline{l}))\subset Sing(C_{\overline{l}})$, for $\overline{l}\in U$, we obtain the result if Case 1 holds above, and, in fact $V_{3}$ has two irreducible components.\\

(!!!!!) (Fixed Singularities 3). Let $L\subset (Sing(C))$ be a generically reducible linear system, (with $2$ irreducible components) of curves of degree $n$. Then, for $\overline{l}\in L$, $C_{\overline{l}}=C_{1}\cup D_{\overline{l}}$, where $C_{1}$ and $D_{\overline{l}}$ are generically irreducible, and the singularities of the generic $D_{\overline{l}}$ are fixed everywhere, centred at $\{p_{l},\ldots,p_{r}\}$.\\

To see $(!!!!!)$, suppose the generic curve $C_{\overline{a}}=C_{1}\cup C_{2}$, $\overline{a}\in U$, where both $C_{1}$ and $C_{2}$ are irreducible curves, of degrees $\{n_{1},n_{2}\}$ with singularities centred at $A=\{q_{l},\ldots,q_{r'}\}$, $B=\{p_{1},\ldots,p_{r}\}$. Choose an independent generic curve $C_{\overline{a}'}$, and consider the $1$-dimensional linear system $l=span(\overline{a},\overline{a}')$. Using the result $(!!)$, we can suppose that $C_{\overline{a}}=(C_{1}\cup D_{\overline{a}})$, $C_{\overline{a}'}=(C_{1}\cup D_{\overline{a}'})$, where $\{D_{\overline{a}},D_{\overline{a}'}\}$ are irreducible of degree $n_{2}$, and belong to a new linear system $L_{2}$. By the result $(\sharp)$, we obtain that the singularities $B=\{p_{l},\ldots,p_{r}\}$ of $C_{2}$ are fixed, for $D_{\overline{l}}$, $\overline{l}\in L_{2}$. As ${\overline{a}'}$ is independent of $\overline{a}$, generic, the fixed singularities, $\{p_{l},\ldots,p_{r}\}$,  are defined over $acl(\overline{a})$, and the conditions that $C_{\overline{l}}\supset C_{1}$ and $\{p_{l},\ldots,p_{r}\}\subset Sing(C_{\overline{l}})$ are closed, the result holds on $L$, as required.\\}).

\end{proof}

\begin{lemma}
\label{harmonic}
Let $C$ be a harmonic curve and let $\{C_{t}:t\in Par_{t}\}$ be a family given as in Definition \ref{harmonic4}. Let $\{(x_{j,j'},y_{j,j'}):1\leq j<j'\leq n\}$ enumerate the points of intersection $l_{j}\cap l_{j'}$, in the coordinate system $(x,y)$. Then, for each $(x_{j,j'},y_{j,j'})$, and $t_{\infty}'\in((\mathcal{V}_{\infty}\cap Par_{t})\setminus\{\infty\})$  there exist exactly $2$ vertical tangents $\{(x_{1,t_{\infty}',j,j'},y_{1,t_{\infty}',j,j'}),(x_{2,t_{\infty}',j,j'},y_{2,t_{\infty}',j,j'})\}$, specialising to $(x_{j,j'},y_{j,j'})$.

\end{lemma}

\begin{proof}
The family $\{C_{t}:t\in Par_{t}\}$ is a particular form of asymptotic degeneration, for which the methods of \cite{dep2} apply. By Lemmas 3.44 and 4.3(iv)(d) of \cite{dep2}, (see notation there), if $t_{\infty}'\in((\mathcal{V}_{\infty}\cap Par_{t})\setminus\{\infty\})$, and $(x_{0},y_{0})\in (l_{j'}\cap l_{j})$, for some $1\leq j<j'\leq n$, then we can find $C_{i,j,\infty}$ and $C_{i',j',\infty}$, for some $1\leq i\leq i'\leq t$, such that $(x_{0},y_{0},z_{0}))\in (C_{i,j,\infty}\cap C_{i',j',\infty'})$, for some $z_{0}\in A^{1}$, and $(x_{1},y_{1},z_{1}))\in (C_{j,t'_{\infty}}\cap C_{j',t'_{\infty}})$, specialising to $(x_{0},y_{0},z_{0})$, where $t'_{\infty}\in\mathcal{V}_{\infty}$. By Lemma 3.44 of \cite{dep2}, as $C_{t_{\infty}'}$ is nonsingular, the corresponding $(x_{1},y_{1})$ defines a vertical tangent of $C_{t_{\infty}'}$. Conversely, if $(x_{1},y_{1})$, defines a vertical tangent of $C_{t_{\infty}'}$, specialising to $(x_{0},y_{0})$, then, by Lemmas 3.44 and 4.3(iii)(b) of \cite{dep2}, there exists a corresponding $(x_{1},y_{1},z_{1})\in (C_{j,t_{\infty}',s_{\infty}'}\cap C_{j',t_{\infty}',s_{\infty}'})$, hence, using Lemma 4.3(iv)(d) of \cite{dep2}, $((x_{0},y_{0})\in (l_{j'}\cap l_{j})$, for some $1\leq j'\neq j\leq n$, $(*)$.\\

Let $W\subset {Par_{t}\setminus\{t_{\infty}\}}\times P^{2}$ be defined by;\\

$W(t,x',y')\equiv [f(t,x',y')\wedge {\partial f_{t}\over\partial x}(x',y')=0]$\\

and let $\overline{W}\subset Par_{t}\times P^{2}$ define the Zariski closure. By $(*)$, we have that the fibre $\overline{W}(t_{\infty}')$ consists of exactly the points $\{(t_{\infty},x_{j,j'},y_{j,j'}):1\leq j<j'\leq n\}$. Suppose, for contradiction, that, for some $(t_{\infty},x_{j_{0},j'_{0}},y_{j_{0},j'_{0}})$, $Mult(\overline{W}/Par_{t})_{(t_{\infty},x_{j_{0},j'_{0}},y_{j_{0},j'_{0}})}\geq 3$.  By the degree-genus formula, Theorem 3.36,  and Severi's Definition 3.33 of genus $g$, see also Theorem 3.36(\dag), in \cite{dep1}, we have that, for $t'\in (Par_{t}\setminus t'_{\infty})$, $g(C_{t'})={(n-1)(n-2)\over 2}$, and $class(C_{t'})=2(g-(1-n))$, hence, $class(C_{t'})=n(n-1)$, $(\dag)$. It follows, using $(*)$, as $Card(\overline{W}(t_{\infty}'))={n(n-1)\over 2}$, that there must exist $(t_{\infty},x_{j_{1},j'_{1}},y_{j_{1},j'_{1}})$, with $Mult(\overline{W}/Par_{t})_{(t_{\infty},x_{j_{1},j'_{1}},y_{j_{1},j'_{1}})}=1$, $(!!!!)$, (\footnote{\label{complex} Let $C^{*}$ denote the dual of $C=C_{0}$, then $deg(C^{*})=cl(C)=n(n-1)$, $cl(C^{*})=deg(C)=n$, using Lemma 5.12 of \cite{dep1} and $(\dag)$ above. Moreover, by Theorem 5.1 of \cite{dep1}, using the fact that $C$ is nonsingular, $i(C^{*})=0$, that is $C^{*}$ has no flexes, $(**)$. Then, using Theorem 4.3 of \cite{dep1};\\

$n+m+2d=n^{2}$ (this generalises to $n+m+(2d+3d'+4d''+\ldots)=n^{2}$, where $\{d,d',d'',\ldots\}$, denotes the number of nodes, triple, quadruple,... $k$-branches (assuming each branch $\gamma$ has $\alpha(\gamma)=1$, we can assume this in the case of $C^{*}$, by $(**)$.   Hence;\\

$(2d+3d'+4d''+\ldots)=n(n-1)-n=n(n-2)$\\

$2(d+d'+d''+\ldots)\leq n(n-2)$\\

$(d+d'+d''+\ldots)\leq {n(n-2)\over 2}$  $(***)$\\

We can add notation to allow for bitangents,tritangents, $k$-tangents,  at inflexions etc, $d_{i_{2},\alpha_{1,i_{2}},\alpha_{2,i_{2}}}$, $d'_{i_{3},\alpha_{1,i_{3}},\alpha_{2,i_{3}},\alpha_{3,i_{3}}}$, and obtain;\\

$(2d+3d'+4d''+\ldots)+(\sum_{i_{2}}(\alpha_{1,i_{2}}+\alpha_{2,i_{2}}))+(\sum_{i_{3}}(\alpha_{1,i_{3}}+\alpha_{2,i_{3}}+\alpha_{3,i_{3}}))+\ldots$\\

$=n(n-1)-n=n(n-2)$\\

If $C$ has no tritangents, we obtain $d={n(n-2)\over 2}$ bitangents, and $d\geq n$, if $n\geq 4$.})\\

Without loss of generality we can assume that the intersection $l_{a}\cap l_{b}$ corresponds to the point $[0:1:0]$ in the coordinate system $x={X\over Z}$, $y={Y\over Z}$, and $l_{a}$ is given by $x=0$, $l_{b}$ is given by $x=1$. The arguments in the paper \cite{dep2}, see especially Lemma 3.44 and Theorem 4.3, apply to the given asymptotic degeneration, with distinct, (\footnote{We can assume that $\eta'_{j,t}(x+1)\neq\eta_{k,t}(x)$, for $1\leq j,k\leq n$, by, wlog, obtaining, using the above argument, that the bitangent $y$-coordinates $\{y(p_{1,1}),\ldots,y(p_{1,n})\}$ are distinct from $\{y(p_{2,1}),\ldots,y(p_{2,n})\}$.}) flashes $\{\eta_{1,t},\ldots,\eta_{n,t}\}$, and $\{\eta'_{1,t},\ldots,\eta'_{n,t}\}$, obtained from applying Newton's theorem along the lines $x=0$ and $x=1$. It follows that, for all $t\in U\subset Par_{t}$, the flashes $\bigcup_{1\leq j\leq n}\eta_{j,t}$ and $\bigcup_{1\leq j\leq n}\eta'_{j,t}$ intersect in finitely many points. Now applying the argument $(*)$, we obtain that, for $\overline{W}$ as above, that, $Mult(\overline{W}/Par_{t})\geq 2$, contradicting $(!!!!)$.\\
\end{proof}

\begin{lemma}
\label{harmonic8}
Let $C$ be a harmonic curve, then there exists a linear system $L$, with $C_{\overline{l}_{0}}=C$, and $C_{\overline{l}_{\infty}}=C_{lines}$, where $C_{lines}$ is a harmonic arrangement. Then, if $\overline{l}\in{\mathcal{V}_{\overline{l}_{\infty}}\setminus\{\overline{l}_{\infty}\}}$, and $p_{i,j}=(l_{i}\cap l_{j})$, $p_{i,j}=(x_{i,j},y_{i,j})$, there exist $\{z^{k}_{i,j}:1\leq k\leq 2\}\subset{\mathcal{V}}_{p_{i,j}}$, with $pr_{x}(z^{k}_{i,j})=x^{k}_{i,j}$ distinct, $z^{k}_{i,j}=(x^{k}_{i,j},y^{k}_{i,j})$, such that $I_{z^{k}_{i,j}}(C_{\overline{l}},x=x^{k}_{i,j})=2$, and, if $z\in C_{\overline{l}}\cap {\mathcal{V}}_{p_{i,j}}$, with $pr_{x}(z)=x^{k}_{i,j}$, then $z=z^{k}_{i,j}$,and, for all $x'\in{\mathcal{V}_{x_{i,j}}\setminus \{x^{k}_{i,j}\}}$, there exist exactly two  $\{z_{x}^{t}:1\leq t\leq 2\}\subset C_{\overline{l}}\cap\mathcal{V}_{p_{i,j}}$, with $pr_{x}(z_{x}^{t})=x'$. If $\overline{l}\in{\mathcal{V}_{\overline{l}_{\infty}}\setminus\{\overline{l}_{\infty}\}}$, and $p\in {C_{lines}\setminus \{p_{i,j}:1\leq i<j\leq n\}}$, $p=(x_{p},y_{p})$, then, for all $x'\in{\mathcal{V}}_{x_{p}}$, there exists a unique $z'\in C_{\overline{l}}\cap{\mathcal{V}}_{p}$, with $pr_{x}(z')=x'$. Finally, if $\overline{l}\in{\mathcal{V}_{\overline{l}_{\infty}}\setminus\{\overline{l}_{\infty}\}}$, and $z\in C_{\overline{l}}$, there exists $p\in C_{lines}$, with $z\in\mathcal{V}_{p}$.
\end{lemma}

\begin{proof}
The existence of $L$ follows from the proof of footnote \ref{variation}. By Lemma \ref{harmonic}, if $\overline{l}\in{\mathcal{V}_{\overline{l}_{\infty}}\setminus\{\overline{l}_{\infty}\}}$, there exist exactly two vertical tangents $\{z^{k}_{i,j}:1\leq k\leq 2\}\subset{\mathcal{V}}_{p_{i,j}}$. Consider the cover $F\subset P^{2}\times L\times P^{1}$, defined by;\\

$F(p,\overline{l},x')\equiv (p\in C_{\overline{l}}\cap x=x')$\\

We claim that $Mult_{(p_{i,j},lines,x_{i,j})}(F/L\times P^{1})=2$, $(*)$. [Considering the $g^{1}_{n}$, on $x=x_{i,j}$, we have, if $p_{i,j}\notin Base(g^{1}_{n})$, and, using Lemma 2.10 of \cite{dep3}, that, $pr_{x}(z^{k}_{i,j})\cap \{x_{i,j}\}=\emptyset$, hence, we can assume that $\{x^{k}_{i,j}:1\leq k\leq 2\}\cap\{x_{i,j}\}=\emptyset$, or say $z^{1}_{i,j}=p_{i,j}$, $(**)$ , $(\dag)$]. Considering the $g^{1}_{n}$, on $x=x_{i,j}$, we have that $I_{p_{i,j}}(C_{lines},x=x_{i,j})=2$, hence, we have that there exist at most $2$ points $\{p'_{i,j},p''_{i,j}\}\subset (\mathcal{V}_{p_{i,j}}\cap C_{\overline{l}})$, with $pr_{x}(p'_{i,j})=pr_{x}(p''_{i,j})=x_{i,j}$. If $Mult_{(p_{i,j},lines,x_{i,j})}(F/L\times P^{1})\geq 3$, then, using summability of specialisation, we have that;\\

$Mult_{(p'_{i,j},\overline{l},x_{i,j})}(F(\overline{l})/P^{1})+Mult_{(p''_{i,j},\overline{l},x_{i,j})}(F(\overline{l})/P^{1})\geq 3$\\

Hence, wlog $I_{p'_{i,j}}(C_{\overline{l}},x=x_{i,j})\geq 2$, implying that $p'_{i,j}$, $(?)$, is a vertical tangent. This contradicts the assumption $(\dag)$, if $(**)$ fails,  as $pr_{x}(p'_{i,j})=x_{i,j}$. If $(**)$ holds, then we must have that $p'_{i,j}=p''_{i,j}=p_{i,j}$, and $I_{p_{i,j}}(C_{\overline{l}},x=x_{i,j})\geq 3$, giving $I_{p_{i,j}}(C_{lines},x=x_{i,j})\geq 3$, which is not the case. If $Mult_{(p_{i,j},lines,x_{i,j})}(F/L\times P^{1})=1$, then clearly, again by summability, $Mult_{(p_{i,j},lines,x_{i,j})}(F(lines)/P^{1})=1$, contradicting the fact that $I_{p_{i,j}}(C_{lines},x=x_{i,j})=2$, giving $(*)$. Suppose that $x'\in ({\mathcal{V}_{x_{i,j}}\setminus\{x^{k}_{i,j}\}})$, and there exists a single $z\in(\mathcal{V}_{p_{i,j}}\cap C_{\overline{l}})$, with $pr_{x}(z)=x'$. As $I_{z}(C_{\overline{l}},x=x')=1$, we have, for generic $(\overline{l},x'')\in\mathcal{V}_{(\overline{l},x)}$, that there exists a single $z\in(\mathcal{V}_{p_{i,j}}\cap C_{\overline{l}})$, with $pr_{x}(z)=x''$, contradicting $(*)$. Similarly, we can exclude $Card(\mathcal{V}_{p_{i,j}}\cap C_{\overline{l}}\cap pr_{x}^{-1}(x'))\geq 3$, in which case, we obtain that $Mult_{(p_{i,j},lines,x_{i,j})}(F/L\times P^{1})\geq 3$, contradicting $(*)$. If, $Card(C_{\overline{l}}\cap pr^{-1}(x^{k}_{i,j})\cap\mathcal{V}_{p_{i,j}})\geq 2$, then, as $I_{z^{k}_{i,j}}(C_{\overline{1}},x=x^{k}_{i,j})=2$, we have that
$Mult_{(p_{i,j},lines,x_{i,j})}(F/L\times P^{1})\geq 2+1=3$, contradicting $(*)$, in particular it follows that $x^{1}_{i,j}\neq x^{2}_{i,j}$. We have that $Mult_{p,lines,x_{p}}(F/L\times P^{1})=1$, as considering the $g_{n}^{1}$ on $x=x_{p}$, and, using the fact that $\{x=x_{p},l_{j}\}$ intersect transversely, where $x_{p}\in l_{j}$, we have that there exists a unique $y''\in{\mathcal{V}_{y_{p}}}$, with $(x_{p},y'')\in C_{\overline{l}}\cap (x=x_{p})$. As $(x_{p},y'')$ does not define a vertical tangent, we have that $I_{(x_{p},y'')}(C_{\overline{l}},x=x_{p})=1$, hence, for generic $(\overline{l},x')\in\mathcal{V}_{(lines,x_{p})}$, there exists a unique $y'\in{\mathcal{V}_{y_{p}}}$ with $(x',y'')\in C_{\overline{l}}\cap x=x'$. Taking $z'=(x',y')$ gives the required result. To see the final part, observe that the variety $Z=\{(\overline{l}',z')\in L\times P^{2},z\in C_{\overline{l}}\}$ is closed, hence, if $(\overline{l},z)\in Z$, with $\overline{l}\in\mathcal{V}_{\overline{l}_{\infty}}$, then its specialisation $(\overline{l}_{\infty},p)\in Z$ as well. By definition, $p\in C_{lines}$.

\end{proof}

\begin{lemma}
\label{harmonic7}
Let hypotheses be as in Lemma \ref{harmonic8}, and assume that $n$ is odd, then there exists an $\epsilon''>0$, and a ball $B(\overline{l}_{\infty},\epsilon'')$ such that for all $\overline{l}\in {(B(\overline{l}_{\infty},\epsilon'')\cap L)\setminus \{\overline{l}_{\infty}\}}$, $C_{\overline{l}}$ is topologically equivalent to a sphere with $g$ attached handles, where
$g={(n-1)(n-2)\over 2}$. In particular, Severi's definition of genus $g$ coincides with the topological definition, see \cite{dep1}.

\end{lemma}

\begin{proof}
Using the result of Lemma \ref{harmonic8}, and Theorem 17.1 of \cite{dep4}, we have, for all infinite $n\in {^{*}\mathcal{N}}$ and $\delta'>0$ standard, that
 the statements $D(n,p_{i,j}),E(n,\delta'), F(n)$ hold;\\

$D(n,p_{i,j})\equiv (\forall{\overline{l}}\in B(\overline{l}_{\infty},{1\over n}))(\forall x\in B(x_{i,j},{1\over n}))[x\notin S(\overline{l})\rightarrow \exists^{=2}y((x,y)\in$\\

\indent \ \ \ \ \ \ \ \  $C_{\overline{l}}\cap B(p_{i,j},{1\over n}))\wedge x\in S(\overline{l})\rightarrow \exists^{=1}y((x,y)\in C_{\overline{l}}\cap B(p_{i,j},{1\over n}))]$\\

$F(n)\equiv (\forall{\overline{l}}\in B(\overline{l}_{\infty},{1\over n}))(\forall z\in C_{\overline{l}})(\exists{=1}w\in C_{lines})(z\in B(w,{1\over n}))$\\

$E(n,\delta')\equiv (\forall{\overline{l}}\in B(\overline{l}_{\infty},{1\over n}))(\forall z\in {C_{lines}\setminus\bigcup_{1\leq i<j\leq n'}B(p_{i,j},\delta')})(\forall x'\in B(x_{z},{1\over n}))(\exists^{=1}y((x',y)\in C_{\overline{l}}\cap B(z,{1\over n})))$\\

where $S$ is defined by $S(\overline{l},x')\equiv (\exists y')(f(x',y',{\overline{l}})=0\wedge{\partial f\over \partial x}(x',y',,{\overline{l}})=0)$   By underflow, see \cite{cut}, the statements hold for all $n\in\mathcal{N}$, with $n\geq n_{0}$. In particular, taking $\epsilon>0$, such that $B(p_{i,j},\epsilon)\cap \{p_{i',j'}:(i',j')\neq (i,j)\}=\emptyset$, we obtain that;\\

(i). For all $\overline{l}\in B'(\overline{l}_{\infty},\epsilon''')$, $pr_{x}:(C_{\overline{l}}\cap B(p_{i,j},\epsilon'''))\rightarrow B(x_{i,j},\epsilon''')$ is a double cover, ramified at two distinct points $\{z^{1}_{i,j}(\overline{l}),z^{2}_{i,j}(\overline{l})\}$, $\epsilon'''\leq\epsilon$, (\footnote{\label{loop} Observe that we can then find a closed loop $S_{i,j,\overline{l}}\subset (C_{\overline{l}}\cap B(p_{i,j},\epsilon'''))$, passing through $\{z^{1}_{i,j}(\overline{l}),z^{2}_{i,j}(\overline{l})\}$, such that $(C_{\overline{l}}\cap B(p_{i,j},\epsilon'''))\cong Ann^{1}\sqcup_{S_{i,j,\overline{l}}} Ann^{2}$, where $\{Ann^{1},Ann^{2}\}$ are complex annuli joined along the loop $S_{i,j,\overline{l}}$.}).\\

(ii). For all $\overline{l}\in B'(\overline{l}_{\infty},\epsilon''')$, and $p\in C_{\overline{l}}$, there exists a unique $w\in C_{lines}$, with $p\in B(w,\epsilon''')$, $\epsilon'''\leq\epsilon$.\\

(iii). For $p\in({C_{lines}\setminus\bigcup_{1\leq i<j\leq n}B(p_{i,j},\epsilon)})$,  there exists $\epsilon'>0$, such that, for all $\overline{l}\in B'(\overline{l}_{\infty},\epsilon''')$, $\epsilon'''\leq\epsilon'$, $pr_{x}:(C_{\overline{l}}\cap B(p,\epsilon'''))\rightarrow B(x_{p},\epsilon''')$ is an isomorphism.\\

Observing that $|p_{i,i+1}-p_{i+1,i+2}|=2sin({2\pi\over n})$, we let $\epsilon''={1\over 3}min(\epsilon,\epsilon')$ and $m_{\epsilon''}=[{2sin({2\pi\over n})+\epsilon''\over \epsilon''}]$. For $0\leq i\leq n-1$, $0\leq s\leq m_{\epsilon''}-1$ we let $p_{i,i+1,s}=p_{i,i+1}+s({p_{i+1,i+2}-p_{i,i+1}\over m_{\epsilon''}})$, and $S=\{s:0\leq s\leq m_{\epsilon''}:B(p_{i,i+1,s},\epsilon'')\cap B(p_{i,i+1},\epsilon)=\emptyset\}$. Clearly $S\subset [0,m_{\epsilon''}-1]$ is an interval, and we let $s_{1}=min(S)-1$, $s_{2}=max(S)+1$. Then, for the real line segment $l^{\mathcal{R}}_{p_{i,i+1},p_{i+1,i+2}}=\{tp_{i,i+1}+(1-t)p_{i+1,i+2}:0\leq t\leq 1\}$, we have $l^{\mathcal{R}}_{p_{i,i+1},p_{i+1,i+2}}\subset \bigcup_{s_{1}\leq s\leq s_{2}}B(p_{i,i+1,s},\epsilon'')\cup B(p_{i,i+1},\epsilon)\cup B(p_{i+1,i+2},\epsilon)$. We then have that, for $\overline{l}\in B'(\overline{l}_{\infty},\epsilon'')$;\\

 $(C_{\overline{l}}\cap(\bigcup_{0\leq i\leq n-1}[B(p_{i,i+1},\epsilon)\cup\bigcup_{s_{1}\leq s\leq s_{2}}B(p_{i,i+1,s},\epsilon'')]\cong T_{1}^{n}$\\

where $T_{1}^{n}$ is an $n$-holed torus. Observe that, using Lemma \ref{equal}, for $0\leq i\leq n-1$, $2\leq k\leq {n-1\over 2}$;\\

$|p_{i,i+1}-p_{i,i+k}|={sin({\pi\over 2}(1-{2\over n}))-sin({\pi\over 2}(1-{2k\over n}))\over cos({\pi\over 2}(1-{2\over n}))sin({\pi\over 2}(1-{2k\over n}))}$\\

For $2\leq k\leq {n-1\over 2}$, we let $m_{k,\epsilon''}={|p_{i,i+k}-p_{i,i+k-1}|+\epsilon''\over \epsilon''}$, and, for $0\leq s\leq m_{k,\epsilon''}-1$, we let $p'_{i,i+k-1,s}=p_{i,i+k-1}+s({p_{i,i+k}-p_{i,i+k-1}\over m_{k,\epsilon''}})$. We let $S_{k}=\{s:0\leq s\leq m_{k,\epsilon''}:B(p_{i,i+k-1,s},\epsilon'')\cap B(p_{i,i+k-1},\epsilon)=\emptyset\}$. Again $S_{k}\subset [0,m_{k,\epsilon''}-1]$ is an interval, and we let $s_{k,1}=min(S_{k})-1$, $s_{k,2}=max(S_{k})+1$. Then, for the real line segment $l^{\mathcal{R}}_{p_{i,i+k-1},p_{i,i+k}}=\{tp_{i,i+k-1}+(1-t)p_{i,i+k}:0\leq t\leq 1\}$, we have $l^{\mathcal{R}}_{p_{i,i+k-1},p_{i,i+k}}\subset \bigcup_{s_{k,1}\leq s\leq s_{k,2}}B(p_{i,i+1,s},\epsilon'')\cup B(p_{i,i+k-1},\epsilon)\cup B(p_{i,i+k},\epsilon)$. It is then clear that, for $\overline{l}\in B'(\overline{l}_{\infty},\epsilon'')$;\\

 $(C_{\overline{l}}\cap(\bigcup_{0\leq i\leq n-1}[B(p_{i,i+1},\epsilon)\cup B(p_{i,i+2},\epsilon)\cup\bigcup_{s_{1}\leq s\leq s_{2}}B(p_{i,i+1,s},\epsilon'')\cup\bigcup_{s_{2,1}\leq s\leq s_{2,2}}B(p'_{i,i+1,s},\epsilon'')]))\cong T_{1,n}^{n}$\\

where $T_{1,n}$ is a torus with $n$ attached handles, and $T_{1,n}^{n}$ is a $T_{1,n}$ with $n$-holes. Repeating the process $l$ times, we obtain that, for $\overline{l}\in B'(\overline{l}_{\infty},\epsilon'')$;\\

 $(C_{\overline{l}}\cap(\bigcup_{0\leq i\leq n-1}[B(p_{i,i+1},\epsilon)\cup \bigcup_{2\leq k\leq 2+(l-1)}B(p_{i,i+k},\epsilon)\cup\bigcup_{s_{1}\leq s\leq s_{2}}B(p_{i,i+1,s},\epsilon'')\cup\bigcup_{2\leq k\leq 2+(l-1),s_{k,1}\leq s\leq s_{k,2}}B(p'_{i,i+k,s},\epsilon'')]))\cong T_{1,nl}^{n}$, (\footnote{\label{attachments} At each stage, the loop $S_{i,j,\overline{l}}$, corresponding to the attachment of the new handle around $p_{i,j}$, should be thought of as connecting annuli on the handles corresponding to $\{p_{i,j-1},p_{i+1,j}\}$. The number of holes $n$ is unchanged, as the loop $S_{i,j,\overline{l}}$ blocks any new passages along the surface. We then obtain a $T_{1,n(l-1),n}^{n}$, where $T_{1,n(l-1),n}$ is a $T_{1,n(l-1)}$ with $n$ attached handles. Sliding these attachments to the main body, $T_{1,n(l-1),n}\cong T_{1,nl}$, giving the required $T_{1,nl}^{n}$.}).\\

 Repeating the process $({n-1\over 2}-2)+1={n-3\over 2}$ times, and, using Lemma \ref{positions}, we obtain that, for $\overline{l}\in B'(\overline{l}_{\infty},\epsilon'')$;\\

  $(C_{\overline{l}}\cap(\bigcup_{0\leq i\leq n-1}[B(p_{i,i+1},\epsilon)\cup \bigcup_{2\leq k\leq {n-1\over 2}}B(p_{i,i+k},\epsilon)\cup\bigcup_{s_{1}\leq s\leq s_{2}}B(p_{i,i+1,s},\epsilon'')\cup\bigcup_{2\leq k\leq {n-1\over 2},s_{k,1}\leq s\leq s_{k,2}}B(p'_{i,i+k,s},\epsilon'')]))\cong T_{1,n({n-3\over 2})}^{n}=T_{1,g-1}^{n}$\\

where $g={(n-1)(n-2)\over 2}$.\\

Finally, let $\{p_{i,\infty}:1\leq i\leq n\}$ denote the points at $\infty$ of $C_{lines}$. Changing coordinates to $(x',y')$ with $\{p_{i,\infty}:1\leq i\leq n\}$ in finite position, say at $\{(0,y'_{i}):0\leq i\leq n\}$, we can assume that for all $\overline{l}\in B'(\overline{l}_{\infty},\epsilon_{0})$, $1\leq i\leq n$, $pr_{x'}:C_{\overline{l}}\cap B((0,y'_{i}),\epsilon_{0})\rightarrow B(0,\epsilon_{0})$ is an isomorphism, (\footnote{\label{infinity} Strictly speaking, we should include this coordinate change and the fixed points at infinity in the definitions of $D(n,p_{i,j})$}). For $\overline{l}\in B(lines,\epsilon_{0})$, we let $D'_{\overline{l}}=C_{\overline{l}}\cap (pr_{x'}^{-1})(B(0,\epsilon_{0}))$, $D_{i,\overline{l}}=l_{i}\cap (pr_{x'}^{-1})(B(0,\epsilon_{0}))$, $D'_{i,\overline{l}}=C_{\overline{l}}\cap B((0,y'_{i}),\epsilon_{0})$, so $D_{i,\overline{l}}\cong D'_{i,\overline{l}}$. Choose a standard $\lambda>0$ such that, for $1\leq i\leq n$, $(D(\overline{0},\lambda)\cap D_{i,\overline{l}})\neq\emptyset$, in coordinates $(x,y)$. Then it follows, taking $\epsilon<<\epsilon_{0}$, that, $C_{\overline{l}}= (D(\overline{0},\lambda)\cap C_{\overline{l}})\cup \bigcup_{1\leq i\leq n}D'_{i,\overline{l}}$, for $\overline{l}\in B'(lines,\epsilon)$, $(\dag)$. We can obviously assume that $\bigcup_{1\leq i<j\leq n}B(p_{i,j},\epsilon'')\subset D(\overline{0},\lambda)\subset D(\overline{0},\lambda+\epsilon'')$, and, hence, that;\\

 $\bigcup_{1\leq i\leq n,s_{1}\leq s\leq s_{2}}B(p_{i,i+1,s},\epsilon'')\cup\bigcup_{1\leq i\leq n,2\leq k\leq {n-1\over 2},s_{k,1}\leq s\leq s_{k,2}}B(p'_{i,i+k,s},\epsilon'')$\\

 $\subset D(\overline{0},\lambda)\subset D(\overline{0},\lambda+\epsilon'')$\\

As\\

 ${(\overline{D(\overline{0},\lambda+\epsilon'')}\cap l_{i})\setminus[\bigcup_{1\leq i\leq n,s_{1}\leq s\leq s_{2}}(B(p_{i,i+1,s},\epsilon'')\cap l_{i})\cup}$\\

 $\bigcup_{1\leq i\leq n,2\leq k\leq {n-1\over 2},s_{k,1}\leq s\leq s_{k,2}}(B(p'_{i,i+k,s},\epsilon'')\cap l_{i})]$\\

 is compact, for each $1\leq i\leq n$, we can find a finite set $Q_{i}$, $|Q_{i}|=P$, with $Q_{i}\subset l_{i}\cap \overline{D(\overline{0},\lambda+\epsilon'')}$, distinct from $W_{i}=\{p_{i,i+1,s},p'_{i,i+k,s},p_{i,j}:j\neq i,2\leq k\leq {n-1\over 2},s_{1}\leq s\leq s_{2},s_{k,1}\leq s\leq s_{k,2}\}$, such that $(l_{i}\cap{\overline{D(\overline{0},\lambda+\epsilon'')}})=\bigcup_{p\in Q_{i}\cup V_{i}}B(p,\epsilon'')\cup\bigcup_{p\in {W_{i}\setminus V_{i}}}B(p,\epsilon)$, where\\
$V_{i}=({W_{i}\setminus \{p_{i,j}:j\neq i\}})$, $(*)$. Using $(ii)$, we have, as $\epsilon''<\epsilon$, that if $w\in (D(\overline{0},\lambda)\cap C_{\overline{l}})$, there exists $w'\in C_{lines}$, with $w\in B(w',\epsilon'')$. By the triangle inequality, we have that $w'\in D(0,\lambda+\epsilon'')$, hence, $w'\in (l_{i}\cap{\overline{D(\overline{0},\lambda+\epsilon'')}})$, for some $1\leq i\leq n$, $(**)$. It follows, by $(*),(**)$, that $(D(\overline{0},\lambda)\cap C_{\overline{l}})\subset \bigcup_{1\leq i\leq n,p\in Q_{i}}(C_{\overline{l}}\cap B(p,2\epsilon''))\cup{1\leq i\leq n,p\in V_{i}}(C_{\overline{l}}\cap B(p,\epsilon''))\bigcup\bigcup_{1\leq i\leq n,p\in {W_{i}\setminus V_{i}}}(C_{\overline{l}}\cap B(p,\epsilon))$. By $(iii)$, as $2\epsilon''<\epsilon'$, we have that, for $1\leq i\leq n$, $p\in Q_{i}$, $pr_{x}:(C_{\overline{l}}\cap B(p,2\epsilon''))\rightarrow B_{x_{p},2\epsilon''}$ is an isomorphism, $(***)$. Moreover, for any such disc $(C_{\overline{l}}\cap B(p,2\epsilon''))$, there exists a finite chain $\{t_{i}:1\leq i\leq r(p)\leq P\}$, with the property that $t_{1}=p$, $(C_{\overline{l}}\cap B(t_{i},2\epsilon'')\cap (C_{\overline{l}}\cap B(t_{i+1},2\epsilon'')))\neq\emptyset$, $(C_{\overline{l}}\cap B(t_{r(p)},2\epsilon'')\cap (C_{\overline{l}}\cap (B(p,\epsilon)\cup B(q,\epsilon'')))\neq\emptyset$, some $p\in {W_{i}\setminus V_{i}}$, $q\in V_{i}$, $(****)$. We let;\\

 $K_{1}=\bigcup_{1\leq i\leq n}W_{i}$\\

 $C_{1,\overline{l}}=\bigcup_{1\leq i\leq n,p\in V_{i}}(C_{\overline{l}}\cap B(p,\epsilon''))\cup\bigcup_{1\leq i\leq n,p\in{W_{i}\setminus V_{i}}} (C_{\overline{l}}\cap B(p,\epsilon))$\\

 and inductively, define;\\

 $K_{j+1}=K_{j}\cup \bigcup_{1\leq i\leq n}\{p\in Q_{i}:B(p,2\epsilon'')\cap C_{j,\overline{l}}\neq\emptyset\}$\\

 $C_{j+1,\overline{l}}=C_{j,\overline{l}}\cup\bigcup_{p\in {K_{j+1}\setminus K_{j}}}(C_{\overline{l}}\cap B(p,\epsilon''))$\\

 By $(****)$, we have that $C_{\overline{l}}\cap B(\overline{0},\lambda)=C_{B,\overline{l}}$, for some $B\leq P$, and, by $(\dag)$, $C_{\overline{l}}=C_{B+1,\overline{l}}=C_{B,\overline{l}}\cup\bigcup_{1\leq i\leq n}D'_{i,\overline{l}}$.\\

We have that $C_{j,\overline{l}}\subset C_{j+1,\overline{l}}$, for $1\leq j\leq B$, and $C_{1,\overline{l}}\cong T^{n}_{1,g-1}$. It follows easily, as each $C_{j,\overline{l}}\subset P^{2}$ is open in the complex topology, for $1\leq j\leq B$, and $C_{B+1,\overline{l}}$ is closed, nonsingular, that $C_{\overline{l}}$ is isomorphic (topologically) to $T_{1,g-1}=S_{g}$, where $S_{g}$ is a sphere with $g$ attached handles.\\

The final claim follows from the proof of the degree-genus formula, with Severi's definition of genus, see \cite{dep1}.\\

\end{proof}

\begin{rmk}
\label{evenodd}
The case when $n$ is even, is left to the reader, the idea being simply to change coordinates, so that there are no intersections $p_{i,j}=(l_{i}\cap l_{j})$ at infinity, and apply the methods of Section 2.

\end{rmk}

\begin{rmk}
\label{degreegenus}
This gives an alternative proof of the (topological) degree-genus formula, see 4.1.1 of \cite{K}, another proof can be found in 4.1.2 of \cite{K}.

\end{rmk}

\end{section}

\end{document}